\documentclass[10pt,a4paper, reqno ]{amsart}
\usepackage[foot]{amsaddr}
\usepackage[toc,page]{appendix}
\usepackage[utf8]{inputenc}
 \usepackage{hyperref}
\usepackage{amsmath}
\usepackage{amsfonts}
\usepackage{mathtools}
\usepackage{graphicx}

\usepackage{tikz}

  \usepackage{geometry}
\usepackage{cite}
\usepackage{amsthm}
\usepackage[yyyymmdd,hhmmss]{datetime}
\usepackage{stmaryrd}

\newtheorem{theorem}{Theorem}[section]
\newtheorem{lemma}[theorem]{Lemma}
\newtheorem{remark}[theorem]{Remark}

\usepackage{amssymb}
\newcommand{\R}{\mathbb{R}  }

\newcommand{\pfeil}{ \rightarrow }

\newcommand{\into}{\hookrightarrow}

\renewcommand{\phi}{\varphi}
\renewcommand{\div}{\operatorname{div}}
\newcommand{\ljump}{ \llbracket  }
\newcommand{\rjump}{ \rrbracket }

\numberwithin{equation}{section}

\title[Stefan problem with $90^\circ$ contact angle]{Strong solutions to the Stefan problem with Gibbs-Thomson correction and boundary contact}

\author[Maximilian Rauchecker]{Maximilian Rauchecker}
\address{Maximilian Rauchecker, Institut f\"ur Angewandte Analysis, Universit\"at Ulm, 89069 Ulm, Germany}

\numberwithin{equation}{section} 

\begin{document}
\maketitle
\begin{abstract}
We prove existence and uniqueness of strong solutions to the two-phase Stefan problem with Gibbs-Thomson law where the free interface forms a ninety degree contact angle with the fixed boundary. We also discuss existence of global solutions and convergence to equilibria.
\end{abstract}
\section{Introduction}
We consider the Stefan problem with Gibbs-Thomson law, where the free interface forms a ninety degree angle to the boundary,
\begin{equation} \label{34905384953745034875} \begin{cases}
\begin{alignedat}{2}
\partial_t u - \Delta u &= 0, &\text{in } \Omega \backslash \Gamma(t), \\
 \ljump u \rjump =0, \quad u|_\Gamma &= \sigma H_\Gamma &\text{on } \Gamma(t), \\
(n_{\partial \Omega} | \nabla u ) &= 0,&\text{on } \partial\Omega \backslash \partial\Gamma(t), \\
V_\Gamma &= - \ljump n_\Gamma \cdot \nabla u \rjump,  \quad\quad\quad  &\text{on }  \Gamma(t), \\
(n_\Gamma | n_{\partial\Omega} ) &= 0, &\text{on } \partial \Gamma(t), \\
u|_{t=0} &= u_0, &\text{in } \Omega \backslash \Gamma(0), \\
\Gamma|_{t=0} &= \Gamma_0.
\end{alignedat} \end{cases}
\end{equation}
Here, $\Omega \subset \R^d $, $d= 2,3$, is a bounded, smooth domain with exterior unit normal vector field $n_{\partial \Omega}$. We assume that the domain can be decomposed as $\Omega = \Omega^+ (t) \dot \cup \mathring \Gamma (t) \dot \cup \Omega^- (t)$, where $\mathring \Gamma (t)$ denotes the interior of the free interface $\Gamma (t)$, a $(d-1)$-dimensional submanifold with boundary. Furthermore, we assume $\Omega^\pm (t)$ to be both connected. The unit normal vector field on $\Gamma(t)$ pointing from $\Omega^- (t)$ to $\Omega^+ (t)$ is denoted by $n_\Gamma$. By $H$ and $V$ we denote the mean curvature and the normal velocity of the free interface, respectively. The jump of a quantity across the interface in direction of the normal $n_\Gamma$ is denoted by $\ljump \cdot \rjump$. Moreover, $\sigma > 0$ is a surface tension constant.

For simplicity we consider the case where the domain is of cylindrical type, that is, $\Omega = \Sigma \times (L_1, L_2)$, where $-\infty < L_1 < 0 < L_2 < \infty$, and $\Sigma \subset \R^{d-1}$ is a bounded, smooth domain. Let $S_1 := \partial \Sigma \times (L_1,L_2)$ and $S_2 := \Sigma \times \{ L_1, L_2 \}$.

\section{Transformation to fixed reference surface}
In this section we transform the Stefan problem \eqref{34905384953745034875} with free interface $\Gamma(t) $ to a fixed reference configuration. To this end we construct a Hanzawa transformation as follows. We follow the ideas of \cite{wilkehabil}. For a suitable construction in a non-cylinder case with curved boundary we refer to \cite{mulsekpaper12}, \cite{mulsek2d}.
To simplify notation, we let $d = 3$.

We asumme that the free interface $\Gamma(t)$ is given as a graph of a height function $h$ over $\Sigma$. More precisely, we assume that there is some $h$ depending on $x' \in \Sigma$ and time $t \geq 0$, such that
\begin{equation}
\Gamma(t) = \Gamma_h (t) := \{ x \in \Sigma \times (L_1,L_2) : x_3 = h(x' , t) , x'  \in \Sigma \}, \quad t \geq 0,
\end{equation}
at least for small times. We then fix a smooth bump function $\chi \in C_0^\infty(\R ; [0,1] )$ such that $\chi (s) = 1$ for $|s| \leq \delta/2$ and $\chi (s) = 0$ for $|s| \geq \delta$. Here, $0 < \delta \leq \min(-L_1,L_2)/3$. Let us now define
\begin{equation} \label{4jjjjjgjfjgjjfgjfjgfj}
\Theta_h : \Omega \times \R_+ \pfeil \Omega, \quad \Theta_h (x,t) := x + \chi(x_3) h(x',t)e_3, 
\end{equation}
where $x = (x', x_3)$.
Then a straightforward calculation shows
\begin{equation}
D \Theta_h =  \begin{pmatrix}
1 & 0 & 0 \\ 0 & 1 & 0 \\
\partial_1 h \chi & \partial_2 h \chi & 1 + h \chi'
\end{pmatrix}.
\end{equation}
If now $h \chi'$ is sufficiently small, $\Theta_h$ is invertible. This is ensured whenever e.g. $
|h|_{\infty} \leq 1/(2 | \chi^\prime|_\infty ) $.
Note that $|\chi'|_\infty$ can be bounded by a constant depending on $\delta$ only.
The inverse is then given by
\begin{equation}
(D\Theta_h)^{-1} = \frac{1}{1+h\chi'}\begin{pmatrix}
1+ h\chi'& 0 & 0 \\ 0 & 1+h\chi' & 0 \\
-\partial_1 h \chi & -\partial_2 h \chi & 1 
\end{pmatrix}.
\end{equation}
For the sequel we fix the cut-off $\chi$ and choose $0 < d_0 < 1/(2|\chi'|_\infty)$ sufficiently small. If now $|h|_{\infty} \leq d_0$, we ensure that the inverse $\Theta_h^{-1} : \Omega \pfeil \Omega$ is well defined and in particular maps the free interface $\Gamma(t)$ to the fixed reference surface $\Sigma$.

Define now the transformed quantity by
\begin{equation}
w (x,t) := u(\Theta_h(x,t), t) ,    \quad x \in \Omega, t \in \mathbb R_+.
\end{equation}
 Define
\begin{equation} \label{34k5kkhjlkh345345345}
D\Theta_h^{-\top} := ((D\Theta_h)^{-1} )^\top = \frac{1}{1+h\chi'}\begin{pmatrix}
1+ h\chi'& 0 & -\partial_1 h \chi \\ 0 & 1+h\chi' & -\partial_2 h \chi \\
 0 & 0 & 1 
\end{pmatrix},
\end{equation}
as well as the transformed operators
\begin{equation*}
\nabla_h  := D\Theta_h^{-\top} \nabla,  \quad \div_h := \operatorname{Tr}(\nabla_h), \quad \Delta_h := \div_h \nabla_h.
\end{equation*}
The upper normal at the free interface $\Gamma_h(t)$ can be written in terms of $h$ by
\begin{equation} \label{8769856}
\nu_{\Gamma(t)} = \frac{(- \nabla h,1)^\top}{\sqrt{ 1+| \nabla h|^2  }}, \quad x' \in \Sigma, t \in \R_+.
\end{equation}
Moreover, the normal velocity satisfies 
\begin{equation}\label{8769856b}
V_{\Gamma(t)} = ( \partial_t h e_3 | \nu_{\Gamma(t)} ) = \frac{\partial_t h	 }{  \sqrt{ 1 + | \nabla h |^2 } }, \quad x' \in \Sigma, t \in \R_+.
\end{equation}
 
We are now able to transform the Stefan problem \eqref{34905384953745034875} with free boundary to $\Omega \backslash \Sigma$. The transformed system reads as 
\begin{equation} \label{978987456885845fgdrerggregrrgrgg9685} \begin{cases}
\begin{alignedat}{2}
\partial_t w -  \Delta w  &= F_w (w,h), & \text{in } \Omega \backslash \Sigma, \\
\ljump w \rjump = 0, \quad  w|_\Sigma - \sigma \Delta_{x'} h &= F_\kappa (h), &\text{on } \Sigma, \\
(n_{\partial\Omega} | \nabla w)  &= F_n(w,h), \quad\quad\quad\quad& \text{on } \partial\Omega \backslash \partial\Sigma, \\
\partial_t h +  \ljump \partial_3 w \rjump &= F_\Sigma(w,h), & \text{on } \Sigma, \\
(( -\nabla h, 0)^\top | n_{\partial \Omega} ) &= 0, &\text{on } \partial \Sigma, \\
w(0) &= w_0, &\text{in } \Omega \backslash \Sigma, \\
h(0) &= h_0, &\text{on } \Sigma,
\end{alignedat} \end{cases}
\end{equation}
where 
\begin{equation} \label{345jjjkj345}
\begin{alignedat}{1}
F_w (w,h) &:=  (\Delta_h - \Delta)w + Dw \cdot \partial_t \Theta_h^{-1}, \\
F_\kappa(h) &:= \sigma \left[ \div \left( \frac{\nabla h}{ \sqrt{ 1 + |\nabla h|^2} } \right) - \Delta h \right], \\
F_n (w,h) &:= (n_{\partial\Omega} | (\nabla-\nabla_h) w) , \\
F_\Sigma (w,h) &:= \ljump \partial_3 w - n_\Gamma \cdot \nabla_h w \rjump + \partial_t h (e_3| e_3 - n_\Gamma ).
\end{alignedat}
\end{equation}
Hereby, $h_0$ is a suitable description of the initial configuration $\Gamma_0$, which we also assume to be a graph over $\Sigma$.

Note that we exploited the fact that on the walls of the cylindrical domain the normal to the boundary $n_{\partial \Omega}$ has last component zero, $(n_{\partial \Omega} | e_3 ) = 0$. This way we obtain the linear boundary condition $\eqref{978987456885845fgdrerggregrrgrgg9685}_5$ for $h$ in the setting of a cylindrical container - in contrast to a general bounded domain where the boundary condition is highly nonlinear, cf. \cite{mulsekpaper12}.

There is also a remark in order regarding the nonlinearity $F_n(w,h)$. By equation \eqref{34k5kkhjlkh345345345} we can write
\begin{equation}
(\nabla - \nabla_h)w = (I - D\Theta_h^{-\top} ) \nabla w = \frac{1}{1+h \chi'} \begin{pmatrix}
-  \partial_1 h \; \partial_3 w \; \chi  \\
-  \partial_2 h \; \partial_3 w  \; \chi \\
-\partial_3 w \; h \; \chi'
\end{pmatrix}.
\end{equation}
Hence, if $h$ satisfies $\eqref{978987456885845fgdrerggregrrgrgg9685}_5$,
\begin{equation}
F_n(w,h) =  \frac{\partial_3 w  \chi}{1+h \chi'}  (n_{\partial \Omega} | (-\nabla h, 0)^\top )_{\mathbb R^3} = 0, \quad \text{on } \partial \Omega \backslash \partial \Sigma.
\end{equation}
Here we used that $(n_{\partial \Omega} | e_3 ) = 0$ on $S_2$ and that $\chi = 0$ in a neighbourhood of $S_1$. This means we may replace $F_n(w,h)$ simply by zero in problem \eqref{978987456885845fgdrerggregrrgrgg9685}. 

As in \cite{mulsekpaper12}, we want to study problem \eqref{978987456885845fgdrerggregrrgrgg9685} in an $L_p-L_q$-setting. This is due to the fact that we need the technical restriction $q <2$ in the spatial integrability to be able to apply certain reflection techniques and avoid additional compatibility conditions. On the contrary, we need $p$ to be sufficiently large to ensure that the height function is $C^2(\Sigma)$ for every time $t \geq 0$. By these reasons we treat the problem in an $L_p-L_q$-theory with $p \not = q$.

\section{Model problems}
In this section we derive maximal regularity of type $L_p-L_q$ for the model problems. There are two relevant model problems in the case of a cylindrical domain: the half-space problem where the flat interface meets the boundary at a ninety degree angle and the quarter space problem without interface but two parts of the boundary. We will start this section however with a characterization of trace spaces for the height function.

Let $1 < p,q < \infty$. In this chapter we consider height functions of class
\begin{equation} \label{234r534534534545jjjjd}
h \in F^{3/2 - 1/2q}_{pq}(0,T; L_q (\Sigma))  \cap F^{1-1/2q}_{pq}(0,T; H^2_q(\Sigma))  \cap L_p(0,T ; W^{4-1/q}_q(\Sigma)).
\end{equation}
For readability, we define $1-1/2q := 1-1/(2q)$ and use similar notation in the following. 
Let us first discuss the traces of $h$ at $t = 0$. 
Using Proposition 5.37 and 5.39 in \cite{kaipdiss},
\begin{equation}
\begin{alignedat}{1}
F^{3/2-1/2q}_{pq}&(0,T; L_q (\Sigma))  \cap F^{1-1/2q}_{pq}(0,T; H^2_q(\Sigma)) \into \\
&\into H^1_p(0,T; W^{2-2/q}_q(\Sigma)).
\end{alignedat}
\end{equation}
By classical results, cf. \cite{amann},
\begin{equation}
\begin{alignedat}{1}
H^1_p(0,T; &W^{2-2/q}_q (\Sigma))    \cap L_p(0,T ; W^{4-1/q}_q(\Sigma)) \into \\ &\into BUC([0,T] ; (W^{4-1/q}_q(\Sigma) ,W^{2-2/q}_q(\Sigma) )_{1-1/p,p} ). \end{alignedat}
\end{equation}
By standard real interpolation,
\begin{equation}
(W^{2-2/q}_q(\Sigma),W^{4-1/q}_q(\Sigma)  )_{1-1/p,p} = B_{qp}^{4-1/q-2/p-1/pq} (\Sigma).
\end{equation}
This space now continuously embeds into $C^2(\Sigma)$, provided
\begin{equation}
4-1/q - 2/p - 1/pq - (n-1)/q > 2, \quad n = 2,3.
\end{equation}
This can easily be achieved by choosing $ q < 2$ close to $2$ and $p < \infty$ large enough. For instance this is ensured for all $p \in (6,\infty)$ and $q \in (19/10, 2)$.
Note that this already implies that the solution space for $h$ in \eqref{234r534534534545jjjjd} continuously embeds into
\begin{equation} \label{34895834573485734857485743}
BUC([0,T] ; C^2 (\Sigma)),
\end{equation}
whenever $p \in (6,\infty)$, $q \in (19/10, 2)$.

\subsection{Trace spaces for the height function.}
\begin{lemma}[Time traces] \label{24345hhhhhsss} Let $p \in (6,\infty), q \in (19/10,2) \cap (2p/(p+1), 2p)$.
The space of traces at $t = 0$ for functions in
\begin{equation} \label{23948534583477384}
F^{3/2-1/2q}_{pq}(0,T; L_q (\R^{n-1}_+))  \cap F^{1-1/2q}_{pq}(0,T; H^2_q(\R^{n-1}_+))  \cap L_p(0,T ; W^{4-1/q}_q(\R^{n-1}_+))
\end{equation}
is given by $B^{4-1/q-2/p}_{qp}(\R^{n-1}_+)$.

 In particular, for every $h_0 \in B^{4-1/q-2/p}_{qp}(\R^{n-1}_+)$ there is some $\bar h$ with regularity \eqref{23948534583477384} satisfying $\bar h(0) = h_0$.
\end{lemma}
\begin{proof}
Firstly, the traces at $t = 0$ are well defined. Let us first show necessity. Consider a function $h$ in \eqref{23948534583477384}. We may extend the function to all of $\R^{n-1}$ and still denote it by $h$. Using Proposition 5.38 in \cite{kaipdiss} we get an embedding
\begin{equation}
\begin{alignedat}{1}
L_p(0,T; &W^{4-1/q}_q (\R^{n-1}))  \cap F^{1-1/2q}_{pq}(0,T; H^2_q(\R^{n-1})) \into \\
&\into H^{1/2}_p(0,T;W^{3-1/q}_q(\R^{n-1})).
\end{alignedat}
\end{equation}
Hence $$h \in H^{1/2}_p(0,T;W^{3-1/q}_q(\R^{n-1})) \cap L_p(0,T ; W^{4-1/q}_q(\R^{n-1})).$$ It is now well known that $( I - \Delta)^{1/2}$ with domain $W^{4-1/q}_q(\R^{n-1})$ has a bounded $H^\infty$-calculus on $W^{3-1/q}_q(\R^{n-1})$, whence we may consider $h$ as a trivial solution of the problem
\begin{equation}
\partial_t^{1/2} h + ( I - \Delta)^{1/2} h = f, \; t > 0, \quad h(0) = h_0,
\end{equation}
where $f := \partial_t^{1/2} h + ( I - \Delta)^{1/2} h  \in L_p( 0,T;W^{3-1/q}_q(\R^{n-1}))$ and $h_0 := h(0)$. Proposition 4.5.14 in \cite{pruessbuch} now gives that
\begin{equation}
h_0 \in ( W^{3-1/q}_q ( \R^{n-1} ) , W^{4-1/q}_q ( \R^{n-1} ) )_{1-1/p\alpha, p}, \quad \alpha = 1/2.
\end{equation} 
An easy calculation shows that then $h_0 \in B^{4-1/q-1/p\alpha}_{qp}(\R^{n-1})$ and this shows necessity.

Let us show the converse direction. So let $h_0 \in B^{4-1/q-2/p}_{qp}(\R^{n-1}_+)$ be given. Again we may extend the function to all of $\R^{n-1}$ without relabeling. Let $A$ be the realization of $I - \Delta$ on $W^{2-1/q}_q(\R^{n-1})$ with natural domain $D(A) = W^{4-1/q}_q(\R^{n-1})$. We know that $A$ enjoys maximal $L_p$-regularity and $-A$ generates an analytic $C_0$-semigroup in $X_0$. We may now solve the equation
\begin{equation}
\partial_t h + A h = 0, \; t >0, \quad h(0) = h_0,
\end{equation}
in an $L_p-L_q$-theory by the function $\bar h := e^{-At} h_0$. We now need to show that $\bar h$ belongs to \eqref{23948534583477384}. Directly from semigroup theory we obtain that
\begin{equation}
\bar h \in H^1_p(0,T; W^{2-1/q}_q (\R^{n-1}))   \cap L_p(0,T ; W^{4-1/q}_q(\R^{n-1})).
\end{equation}
We also have the embedding
\begin{equation}
H^1_p(0,T; W^{2-1/q}_q (\R^{n-1}))   \cap L_p(0,T ; W^{4-1/q}_q(\R^{n-1})) \into F^{1-1/2q}_{pq} (0,T; H^2_q ( \R^{n-1} ) ),
\end{equation}
see Propositions 5.37 and 5.39 in \cite{kaipdiss}. 
Note that $\partial_t \bar h = - A \bar h = - Ae^{-At} h_0 \in F^{1-1/2q}_{pq} (0,T; L_q ( \R^{n-1} ) )$,
whence the proof is complete since $\bar h (0) = h_0$.
\end{proof}
\begin{remark}
Note that for a function $h$ in the above class also $\partial_t h$ has a time trace at $t = 0$. However, the Neumann trace of $h$ has no full time derivative. This will be a consequence of Lemma \ref{34ggg65456546} below, see also Remark \ref{934534534555555h55h5}.
\end{remark}
The following lemma states the optimal regularity for the Neumann traces.
\begin{lemma}[Neumann trace space] \label{34ggg65456546}
Let $p \in (6,\infty), q \in (19/10,2) \cap (2p/(p+1), 2p)$. The Neumann trace $[h \mapsto \nabla h|_{\partial \R^{n-1}_+} ]$ is bounded as a mapping from
\begin{equation} \label{2394jjjjjjjjjjjjg}
{_{00} F}^{3/2-1/2q}_p(0,T; L_q (\R^{n-1}_+))  \cap {_ 0 F}^{1-1/2q}_{pq}(0,T; H^2_q(\R^{n-1}_+))  \cap L_p(0,T ; W^{4-1/q}_q(\R^{n-1}_+))
\end{equation}
to
\begin{equation} \label{23ttttttrjg}
{_0 F}^{\beta(q)}_{pq}(0,T; L_q (\partial \R^{n-1}_+))  \cap {_ 0 F}^{1-1/2q}_{pq}(0,T; W^{1-1/q}_q(\partial \R^{n-1}_+))  \cap L_p(0,T ; W^{3-2/q}_q(\partial \R^{n-1}_+)),
\end{equation}
where
\begin{equation}
\beta (q) = \frac{5}{4} - \frac{1}{q} + \frac{1}{4(3q-1)}.
\end{equation}
Here, the two subscript zeros denote vanishing traces at $ t = 0$, $h(0) = 0$ and $\partial_t h(0) = 0$, one subscript zero means $h(0) = 0$ only. 

Furthermore, there exists a continuous right inverse in the following sense: For every $b$ in \eqref{23ttttttrjg} there exists some $h=E(b)$ in 
\begin{equation} \label{2394jjjjjjjjjjjjg}
{_{0} F}^{3/2-1/2q}_p(0,T; L_q (\R^{n-1}_+))  \cap {_ 0 F}^{1-1/2q}_{pq}(0,T; H^2_q(\R^{n-1}_+))  \cap L_p(0,T ; W^{4-1/q}_q(\R^{n-1}_+))
\end{equation}
such that $\partial_n h|_{\partial \R^{n-1}_+} = b$ and $[b \mapsto E(b)]$ is continuous between the above spaces. In particular, the operator norm of $[b \mapsto E(b)]$ is independent of $T$ (since we restrict to vanishing time traces).
\end{lemma}
\begin{remark}
Note that as $q \uparrow 2$, the regularity index $\beta(q) \pfeil 4/5$, whereas $1-1/2q \pfeil 3/4$. Hence, for $q < 2$ very close to $2$, we can not get rid of the first space in \eqref{23ttttttrjg} since the time regularity there is higher as in the other spaces.
\end{remark}
\begin{remark} \label{934534534555555h55h5}
Note that $\frac{5}{4} - \frac{1}{q} + \frac{1}{4(3q-1)} < 1$ for all $q \in (19/10,2 )$.
\end{remark}
Let us now prove Lemma \ref{34ggg65456546}.
\begin{proof}
By a reflection argument in time (this is possible since we have vanishing traces at $t=0$) and extending the functions to the whole of $\R^{n-1}$ it is enough to consider now a function $h$ in 
\begin{equation} \label{2hhhhghg}
{_{00} F}^{3/2 - 1/2q}_{pq}(\R_+; L_q (\R^{n-1}))  \cap {_ 0 F}^{1-1/2q}_{pq}(\R_+; H^2_q(\R^{n-1}))  \cap L_p(\R_+ ; W^{4-1/q}_q(\R^{n-1})).
\end{equation}
By standard trace theory, we readily get
\begin{equation} \label{2hhggfhfghfghg}
\nabla h \in  {_ 0 F}^{1-1/2q}_{pq}(\R_+; H^1_q(\R^{n-1}))  \cap L_p(\R_+ ; W^{3-1/q}_q(\R^{n-1})).
\end{equation}
We may interpolate the first two spaces and the first and the last in \eqref{2hhhhghg} to the result
\begin{equation}
{_ 0 F}^{5/4 - 1/2q}_{pq}(\R_+; H^1_q(\R^{n-1})), \;  {_ 0 F}^{(3q-1)/(2q+1)}_{pq}(\R_+; H^1_q(\R^{n-1})),
\end{equation}
see \cite{kaipdiss}. 
Since $1 < q <2$, we obtain that the regularity index $5/4 - 1/2q$ is larger than the second one, hence the second space embeds into the first one. This is also natural if we consider the regularity diagram for the Stefan problem below.
Hence
\begin{equation} 
\nabla h \in {_ 0 F}^{5/4 - 1/2q}_{pq}(\R_+; L_q(\R^{n-1})) \cap {_ 0 F}^{1-1/2q}_{pq}(\R_+; H^1_q(\R^{n-1}))  \cap L_p(\R_+ ; W^{3-1/q}_q(\R^{n-1})).
\end{equation}
Standard trace theory in the last two spaces gives
\begin{equation} 
\nabla h|_{\partial \R^{n-1}_+} \in {_ 0 F}^{1-1/2q}_{pq}(\R_+; W^{1-1/q}_q(\R^{n-1}_+))  \cap L_p(\R_+ ; W^{3-2/q}_q(\R^{n-1}_+)).
\end{equation}
We may mow mimic the proof of Theorem B.1 in \cite{mulsekpaper12} and write the space $${_ 0 F}^{\alpha(q)}_{pq}(\R_+; L_q(\R^{n-1}))  \cap L_p(\R_+ ; W^{3-1/q}_q(\R^{n-1}))$$ as an anisotropic Triebel-Lizorkin space, cf. \cite{johnsensickel2008}. Then taking anisotropic traces onto the boundary gives that
\begin{equation} 
\nabla h|_{\partial \R^{n-1}_+} \in {_ 0 F}^{\beta(q)}_{pq}(\R_+; L_q(\R^{n-1}_+))  \cap L_p(\R_+ ; W^{3-2/q}_q(\R^{n-1}_+)),
\end{equation}
where 
\begin{equation}
\beta (q) := \alpha(q) - \frac{ \alpha (q)}{3-1/q} \frac{1}{q}, \quad \alpha (q) := \frac{5}{4} - \frac{1}{2q}.
\end{equation}
Simple calculations then entail that
\begin{equation}
\beta (q) = \frac{5}{4} -\frac{1}{q} + \frac{1}{4(3q-1)}. 
\end{equation}
We have so far shown the first part. It now remains to construct a continuous right inverse. So let us be given some
\begin{equation} \label{345634645645656}
b \in {_ 0 F}^{\beta(q)}_{pq}(\R_+; L_q(\R^{n-1}_+))  \cap L_p(\R_+ ; W^{3-2/q}_q(\R^{n-1}_+)).
\end{equation}
Again mimicking the proof of Theorem B.1 in \cite{mulsekpaper12} we may write this space as an anisotropic Triebel-Lizorkin space and use the results of \cite{johnsensickel2008}. We obtain, given $b$ in \eqref{345634645645656}, a right inverse $E(b)$ of the Neumann trace with
\begin{equation}
E(b) \in {_0 F}^{\gamma(q)}_{pq}(\R_+; L_q(\R^{n-1}_+))  \cap L_p(\R_+ ; W^{4-1/q}_q(\R^{n-1}_+)).
\end{equation}
Here, 
\begin{equation}
\gamma(q) := \frac{1}{3-2/q} (4-1/q) \beta(q) = \frac{5}{3} - \frac{1}{2q} - \frac{1}{12(3q-1)}.
\end{equation}
Now again since $q < 2$, we have $\gamma(q) \geq 3/2 - 1/2q$. Hence by Banach-space valued embeddings for Triebel-Lizorkin spaces, cf. \cite{MR2911497},
\begin{equation}
{_0 F}^{\gamma(q)}_{pq}(\R_+; L_q(\R^{n-1}_+))   \into {_0 F}^{3/2 - 1/2q}_{pq}(\R_+; L_q(\R^{n-1}_+)).
\end{equation}
If now additionally to \eqref{345634645645656}, $b$ also has regularity ${_ 0 F}^{1-1/2q}_{pq}(0,T; W^{1-1/q}_q(\partial \R^{n-1}_+))$,
by standard trace theory in the spatial variable we also obtain that
\begin{equation}
E(b) \in {_ 0 F}^{1-1/2q}_{pq}(0,T; H^{2}_q( \R^{n-1}_+))
\end{equation}
since we already know that $\nabla E(b) |_{\partial \R^{n-1}_+ } = b$. The proof is complete.
\end{proof}
We close this subsection with the regularity diagram for the solution space \eqref{234r534534534545jjjjd}.
\begin{figure}[h]
\centering
\begin{tikzpicture}
\draw[thick,->] (0,0) -- (5,0) node[anchor=north west] {$x$};
\draw[thick,->] (0,0) -- (0,3.6) node[anchor=south east] {$t$};
 \draw (2 cm,2pt) -- (2 cm,-2pt) node[anchor=north] {$2$};
 \draw (3.5 cm,2pt) -- (3.5 cm,-2pt) node[anchor=north] {$4-1/q$};
  \draw (2pt,2.25cm) -- (-2pt,2.25 cm) node[anchor=east] {$3/2-1/2q$};
   \draw (2pt,1.5cm) -- (-2pt,1.5 cm) node[anchor=east] {$1-1/2q$};
   
   \draw[thick,dashed] (0,1.5) -- (4,1.5) ;
   \draw[thick,dashed] (2,0) -- (2,3) ;
      \draw[thick] (0,2.25) -- (2,1.5) ;
       \draw[thick]  (2,1.5) -- (3.5,0) ;
       
       \filldraw[black] (2,1.5) circle [radius=1.35pt];
\end{tikzpicture}
\end{figure}
Here we note that the left line segment has a slope of modulus $1/4$, whereas the right line segment has a slope of modulus $1/2$.
\subsection{The model problem with a flat interface.}
Let us discuss optimal regularity and solvability near the contact line. We will now investigate the model problem consisting of a half-space problem with a flat interface. More precisely, let $n=2,3$, $\Omega := \{ x \in \R^n : x_1 > 0 \}$, and $\Sigma := \Omega \cap \{ x_n = 0 \}.$ Let $\omega > 0$. We consider the linear problem
\begin{equation} \label{43453454445444444454}
\begin{cases} \begin{alignedat}{2}
\partial_t u + \omega u - \Delta u &= f_u, \qquad\qquad\qquad &\text{in } \Omega \backslash \Sigma, \\
(n_{\partial\Omega} | \nabla u) &= f_n, &\text{in } \partial\Omega, \\
\ljump u\rjump =0, \quad u|_\Sigma - \Delta h &= g, &\text{on } \Sigma, \\
(n_{\partial\Omega} | \nabla h) &= b, &\text{in } \partial\Sigma, \\
\partial_t h + \omega h + \ljump \partial_3 u \rjump &= f_h, &\text{on } \Sigma, \\
u(0) &= u_0, &\text{in } \Omega \backslash \Sigma, \\
h(0) &= h_0, &\text{on } \Sigma. \\
\end{alignedat} \end{cases}
\end{equation}
Let $p \in (6,\infty)$, $q \in (19/10,2) \cap (2p/(p+1), 2p)$. We are interested in strong solutions
\begin{equation} \label{43534534534534534534A}
\begin{alignedat}{2}
h \in \mathbb E_h(T) := F^{3/2-1/2q}_{pq}(0,T; L_q (\Sigma))  \cap F^{1-1/2q}_{pq}(0,T; H^2_q(\Sigma))  \cap L_p(0,T ; W^{4-1/q}_q(\Sigma))
\end{alignedat}
\end{equation}
and
\begin{equation}
\begin{alignedat}{2} \label{43534534534534534534B}
u \in \mathbb E_u(T) := H^1_p(0,T; L_q(\Omega)) \cap L_p(0,T; H^2_q(\Omega \backslash \Sigma)) .
\end{alignedat}
\end{equation}
Note that for the case $p = q$ this setting was already considered in \cite{pruessbuch}. We refer to Section 6.6 therein for a motivation for these spaces.
\subsection*{Spaces for the data.}
Suppose we are given a solution $(u,h)$ in the above classes. We want to find necessary conditions for the data. Clearly, $f_u \in L_p(0,T;L_q(\Omega))$. Also, by trace theory,
\begin{equation}
\nabla u |_{\partial\Omega} \in  F^{1/2-1/2q}_{pq} (0,T; L_q(\partial\Omega) ) \cap L_p(0,T; W^{1-1/q}_q(\partial \Omega \backslash \partial \Sigma)).
\end{equation}
Note that the function $\nabla u |_{\partial\Omega}$ is only $W^{1-1/q}_q$ in space. Since $q <2$, it does not possess a trace on $\partial\Sigma$. Hence we obtain that
\begin{equation}
f_n \in  F^{1/2-1/2q}_{pq} (0,T; L_q(\partial\Omega) ) \cap L_p(0,T; W^{1-1/q}_q(\partial \Omega)).
\end{equation}
We also have by classical trace theory that
\begin{equation} \label{3495tttttrhghgh865486}
 u |_{\Sigma} \in  F^{1-1/2q}_{pq} (0,T; L_q(\Sigma) ) \cap L_p(0,T; W^{2-1/q}_q(\Sigma)).
\end{equation}
Since also $\Delta h$ enjoys this regularity, $g$ should belong to this class. Note that we have shown in Lemma \ref{34ggg65456546} that
\begin{align}
\nabla h |_{\partial\Sigma} \in & \; {F}^{5/4- 1/q + 1/4(3q-1)}_{pq}(0,T;  L_q (\partial \Sigma))  \cap {F}^{1-1/2q}_{pq}(0,T; W^{1-1/q}_q(\partial \Sigma)) \\
 &\quad \cap L_p(0,T ; W^{3-2/q}_q(\partial \Sigma)),
\end{align}
whence we take $b$ to belong to this class. 
Furthermore, we directly obtain that
\begin{equation}
f_h \in F^{1/2 - 1/2q}_{pq}(0,T ; L_q(\Sigma)) \cap L_p(0,T; W^{1-1/q}_q(\Sigma)) .
\end{equation}
 By classical real interpolation method, we obtain $u_0 \in W^{2-2/q}_q(\Omega \backslash \Sigma)$ and Lemma \ref{24345hhhhhsss} entails $h_0 \in B^{4-1/q-2/p}_{qp}(\Sigma)$.

\subsection*{Compatibility conditions.}
One important feature of our $L_p-L_q$ theory is that there is no additional compatibility condition for $(f_n,g)$ on $\partial \Sigma$, since the function $\nabla u|_\Sigma$ does not have a well-defined trace on $\partial \Sigma$ since $q < 2$, cf. \eqref{3495tttttrhghgh865486}. Therefore we can reduce the amount of compatibility conditions at the contact line to a minimum. These read
\begin{enumerate}
\item $(n_{\partial\Omega} | \nabla u_0) = f_n(0)$, on $\partial\Omega$,
\item $\ljump u_0 \rjump = 0$, on $\Sigma$,
\item $u_0|_\Sigma - \Delta h_0 = g(0),$ on $\Sigma$,
\item $(n_{\partial\Sigma} | \nabla h_0) = b(0)$, on $\partial\Sigma$,
\item $\ljump \partial_3 u_0 \rjump -f_h (0)\in \mathsf{tr}|_{t=0} \left[  F^{1/2 - 1/2q}_{pq} (0,T; L_q(\Sigma)) \cap L_p(0,T; W^{2-2/q}_q(\Sigma))  \right] $.
\end{enumerate}
Let us comment on these. The first four conditions simply follow by evaluating the respective equations at time $t = 0$. Note that the functions have enough time regularity such that the traces are well defined. Let us explain the last one in more detail. From \eqref{43534534534534534534B}, standard trace theory for $u$ entails
\begin{equation}
\ljump \partial_3 u \rjump -f_h \in  F^{1/2-1/2q}_{pq} (0,T; L_q(\Sigma) ) \cap L_p(0,T; W^{1-1/q}_q(\Sigma)).
\end{equation}
Additionally, however, $\ljump \partial_3 u \rjump - f_h = -\partial_t h - \omega h$ by the equations. Recall that
\begin{equation}
\begin{alignedat}{1}
h \in F^{3/2-1/2q}_{pq}&(0,T; L_q (\Sigma))  \cap F^{1-1/2q}_{pq}(0,T; H^2_q(\Sigma)) \into \\
&\into H^1_p(0,T; W^{2-2/q}_q(\Sigma)),
\end{alignedat}
\end{equation}
 hence $-\partial_t h - \omega h \in L_p(0,T; W^{2-2/q}_q(\Sigma))$. This yields that
 \begin{equation}
 \ljump \partial_3 u_0 \rjump - f_h (0) \in \mathsf{tr}|_{t=0} \left[  F^{1/2 - 1/2q}_{pq} (0,T; L_q(\Sigma)) \cap L_p(0,T; W^{2-2/q}_q(\Sigma))  \right].
 \end{equation}
 We remark that we can explicitly calculate the trace space as in the proof of Lemma \ref{24345hhhhhsss} to the result
\begin{equation} \label{34534453453454354jj}
\ljump \partial_3 u_0 \rjump - f_h(0) \in B^{2-2/q - 4/p}_{qp}(\Sigma).
\end{equation}

\subsection*{Maximal $L_p-L_q$-regularity.} We now want to show that the necessary conditions derived above together with the compatibility conditions are also sufficient. Suppose we are given $(f_u, f_n, g, f_h, b, u_0,h_0)$ satisfying the above conditions. We now want to solve \eqref{43453454445444444454}.

Let us first reduce to $h_0 = 0$ and $\ljump \partial_3 u_0 \rjump - f_h (0) = 0$. The idea stems from Section 6.2 in \cite{pruessbuch}.
So let $h_0 \in B^{4-1/q - 2/p}_{qp}(\R^{n-1})$ and $h_1 := f_h(0) - \ljump \partial_3 u_0 \rjump - \omega h_0 \in B^{2-2/q-4/p}_{qp}(\R^{n-1})$ be the extensions of $h_0$ and $f_h(0) - \ljump \partial_3 u_0 \rjump - \omega h_0$ to $\R^{n-1}$. We define the operators $A := 1 + \omega - \Delta$ and $B := 1+ \omega + \Delta^2$. It is well-known that these are negative generators of exponentially stable analytic $C_0$-semigroups with maximal $L_p$-regularity on $L_q(\R^{n-1})$, hence also on $W^s_q(\R^{n-1})$, $s > 0$.
Let us define
\begin{equation} \label{345345jjjjj345345}
\bar h(t) := (2e^{-At} - e^{-2At})h_0 + (e^{-Bt} - e^{-2Bt})B^{-1} h_1.
\end{equation}
Then clearly $\bar h(0) = h_0$, $(\partial_t + \omega)\bar h(0) = h_1 + \omega h_0$. Note that the function $e^{-At}h_0$ solves the evolution problem
\begin{equation}
\partial_t h + Ah = 0, \; t > 0, \quad h(0) = h_0.
\end{equation}
In the proof of Lemma \ref{24345hhhhhsss} we actually showed that
$
e^{-At}h_0 \in \mathbb E_h(T).
$
It remains to prove that $e^{-Bt} B^{-1} h_1 \in \mathbb E_h(T)$. Consider $B$ on the base space $W^{2-2/q}_q(\R^{n-1})$ with natural domain $W^{6-2/q}_q(\R^{n-1})$. The function $e^{-Bt}B^{-1}h_1$ then solves the evolutionary problem
\begin{equation}
\partial_t h + Bh = 0, \; t > 0, \quad h(0) = B^{-1} h_1.
\end{equation}
By maximal regularity,
\begin{equation}
e^{-Bt} B^{-1} h_1 \in H^1_p(0,T; W^{2-2/q}_q(\R^{n-1})) \cap L_p(0,T; W^{6-2/q}_q(\R^{n-1}) ).
\end{equation}
Note that by construction, $\partial_t e^{-Bt} B^{-1} h_1 = -e^{-Bt} h_1$. Since
\begin{equation}
H^1_p(0,T; W^{2-2/q}_q(\R^{n-1})) \cap L_p(0,T; W^{6-2/q}_q(\R^{n-1}) ) \into F^{1/2 - 1/2q}_{pq}(0,T; H^4_q(\R^{n-1}) ),
\end{equation}
we obtain $-e^{-Bt} h_1 \in F^{1/2 - 1/2q}_{pq}(0,T; L_q(\R^{n-1}) )$, hence $e^{-Bt} B^{-1} h_1 \in \mathbb E_h(T)$.

By subtracting $\bar h$ from problem \eqref{43453454445444444454} we reduce to $h_0 = 0$ and trivialize the last compatibility condition to $\ljump \partial_3 u_0 \rjump - f_h(0) = 0$. Due to this now generated compatibility condition between $u_0$ and $f_h(0)$ we may solve the transmission problem
\begin{equation}
\begin{cases}
\begin{alignedat}{2}
\partial_t \bar u + \bar \omega \bar u - \Delta \bar u &= 0, \quad\quad\quad &\text{in } \R^{n} \backslash \tilde \Sigma, \\
\ljump \bar u \rjump &= 0, &\text{on } \tilde \Sigma, \\
\ljump \partial_3 \bar u \rjump &= f_h, & \text{on } \tilde \Sigma, \\
\bar u(0) &= u_0, &\text{in } \R^{n} \backslash \tilde \Sigma,
\end{alignedat}
\end{cases}
\end{equation}
in an $L_p-L_q$-theory by a function $\bar u \in H^1_p(0,T; L_q(\R^n)) \cap L_p(0,T; H^2_q(\R^n \backslash \tilde \Sigma))$ using Theorem 6.5.1 in \cite{pruessbuch}.
Here again $u_0$ and $f_h$ are extensions and $\tilde \Sigma := \{ x \in \R^n : x_n = 0 \}$.
Surely we can restrict $\bar u$ again back to the half space $\{ x_1 > 0 \}$. We have now so far reduced \eqref{43453454445444444454} to the problem

\begin{equation}  \label{345345345345jjjj5}
\begin{cases} \begin{alignedat}{2}
\partial_t u + \omega u - \Delta u &= f_u, \qquad\qquad\qquad &\text{in } \Omega \backslash \Sigma, \\
(n_{\partial\Omega} | \nabla u) &= f_n, &\text{in } \partial\Omega, \\
\ljump u\rjump =0, \quad u|_\Sigma - \Delta h &= g, &\text{on } \Sigma, \\
(n_{\partial\Omega} | \nabla h) &= b, &\text{in } \partial\Sigma, \\
\partial_t h + \omega h + \ljump \partial_3 u \rjump &= 0, &\text{on } \Sigma, \\
u(0) &= 0, &\text{in } \Omega \backslash \Sigma, \\
h(0) &= 0, &\text{on } \Sigma, \\
\end{alignedat} \end{cases}
\end{equation}
for possibly modified right hand sides which we do not relabel. We want to note at this point that we have vanishing traces as compatibility conditions in \eqref{345345345345jjjj5},
\begin{equation}
f_n (0) = 0, \quad g(0) = 0, \quad b(0) = 0.
\end{equation}
By the theory for elliptic equations, cf. \cite{pruessbuch}, we may find some 
\begin{equation}
\tilde u \in H^1_p(0,T; L_q(\Omega)) \cap L_p(0,T; H^2_q(\Omega ))
\end{equation}
in an $L_p-L_q$-theory solving
\begin{equation} 
\begin{cases} \begin{alignedat}{2}
\partial_t \tilde u  + \tilde \omega \tilde u - \Delta \tilde u &= 0, \qquad\qquad\qquad &\text{in } \Omega, \\
(n_{\partial\Omega} | \nabla \tilde u) &= f_n, &\text{in } \partial\Omega, \\
\tilde u (0) &= 0, &\text{in } \Omega,
\end{alignedat} \end{cases}
\end{equation}
since $f_n(0) = 0$. Hereby we used the fact that $q < 2$ and that therefore $f_n$ does not possess a trace on $\partial\Sigma$.
Also, by Lemma \ref{34ggg65456546} we find some $\tilde h$ in the proper regularity class satisfying
\begin{equation}
(n_{\partial\Sigma} | \nabla \tilde h) = b,\quad  \text{in } \partial\Sigma, \\
\end{equation}
since $b(0) = 0$. Subtracting $(\tilde u, \tilde h)$ we may again reduce the problem to
\begin{equation}  \label{344dfgdfgdfgdfgvvv4}
\begin{cases} \begin{alignedat}{2}
\partial_t u + \omega u - \Delta u &= f_u, \qquad\qquad\qquad &\text{in } \Omega \backslash \Sigma, \\
(n_{\partial\Omega} | \nabla u) &= 0, &\text{in } \partial\Omega, \\
\ljump u\rjump =0, \quad u|_\Sigma - \Delta h &= g, &\text{on } \Sigma, \\
(n_{\partial\Omega} | \nabla h) &= 0, &\text{in } \partial\Sigma, \\
\partial_t h + \omega h + \ljump \partial_3 u \rjump &= -\partial_t \tilde h -\omega \tilde h, &\text{on } \Sigma, \\
u(0) &= 0, &\text{in } \Omega \backslash \Sigma, \\
h(0) &= 0, &\text{on } \Sigma. \\
\end{alignedat} \end{cases}
\end{equation}
We note that in particular
\begin{equation} \label{345345345gg}
 -\partial_t \tilde h -\omega \tilde h \in F^{1-1/2q}_{pq} (0, T; L_q (\Sigma)) \cap L_p(0,T; W^{2-2/q}_q(\Sigma)) ,
\end{equation}
since $\ljump \partial_3 \tilde u \rjump = 0$. Recall that $g$ has regularity
\begin{equation}
g \in {_0 F}^{1-1/2q}_{pq} (0,T; L_q(\Sigma) ) \cap L_p(0,T; W^{2-1/q}_q(\Sigma)).
\end{equation}
It is now an essential feature that $q <2$. 
As in \cite{mulsekpaper12} we reflect $(u,h,f_u,g,-\partial_t \tilde h -\omega \tilde h)$ evenly in $x_1$-direction across the boundary of $\Omega$. We are therefore left to solve a full-space problem with flat interface, where the restriction of the solutions back to $\Omega$ gives back the solution of the original problem \eqref{344dfgdfgdfgdfgvvv4}. This is due to the fact that by the even reflection, the conditions $\eqref{344dfgdfgdfgdfgvvv4}_2$ and $\eqref{344dfgdfgdfgdfgvvv4}_4$ hold automatically. We are left to solve the full-space problem
\begin{equation}  \label{3gghjrrerercchj4}
\begin{cases} \begin{alignedat}{2}
\partial_t u + \omega u - \Delta u &= f_u, \qquad\qquad\qquad &\text{in } \R^n \backslash \tilde \Sigma, \\
\ljump u\rjump =0, \quad u|_{\tilde\Sigma} - \Delta h &= g, &\text{on } \tilde \Sigma, \\
\partial_t h + \omega h + \ljump \partial_3 u \rjump &= -\partial_t \tilde h -\omega \tilde h, &\text{on } \tilde \Sigma, \\
u(0) &= 0, &\text{in } \R^n \backslash \tilde\Sigma, \\
h(0) &= 0, &\text{on } \tilde\Sigma, \\
\end{alignedat} \end{cases}
\end{equation}
where we again for simplicity did not relabel the functions. Let $\tilde h_0 := 0$, $\tilde h_1 := -\partial_t \tilde h (0) - \omega \tilde h(0)$. Recalling \eqref{345345345gg}, we may repeat the first step in the proof and subtract the function
\begin{equation}
\bar {\bar h} (t) := (2e^{-At} - e^{-2At})\tilde h_0 + (e^{-Bt} - e^{-2Bt})B^{-1} \tilde h_1
\end{equation}
constructed in \eqref{345345jjjjj345345}. Then again $\bar {\bar h} \in \mathbb E_h(T)$, $\bar {\bar h}(0) = \tilde h_0 = 0$, $(\partial_t + \omega) \bar {\bar h} (0) = \tilde h_1$, and we reduce to
\begin{equation}  \label{3ffffrrrppppphj4}
\begin{cases} \begin{alignedat}{2}
\partial_t u + \omega u - \Delta u &= f_u, \qquad\qquad\qquad &\text{in } \R^n \backslash \tilde \Sigma, \\
\ljump u\rjump =0, \quad u|_{\tilde\Sigma} - \Delta h &= g, &\text{on } \tilde \Sigma, \\
\partial_t h + \omega h + \ljump \partial_3 u \rjump &= f_h, &\text{on } \tilde \Sigma, \\
u(0) &= 0, &\text{in } \R^n \backslash \tilde\Sigma, \\
h(0) &= 0, &\text{on } \tilde\Sigma, \\
\end{alignedat} \end{cases}
\end{equation}
where $f_h (0) = 0$. This allows then in turn again to solve a transmission problem
\begin{equation}  
\begin{cases} \begin{alignedat}{2}
\partial_t \hat u + \hat \omega \hat u - \Delta \hat u &= 0, \qquad\qquad\qquad &\text{in } \R^n \backslash \tilde \Sigma, \\
\ljump \hat u \rjump &= 0, &\text{on } \tilde \Sigma, \\
 \ljump \partial_3 \hat u \rjump &= f_h, &\text{on } \tilde \Sigma, \\
\hat u (0) &= 0, &\text{in }  \R^n \backslash \tilde \Sigma, \\
\end{alignedat} \end{cases}
\end{equation}
by a function $\hat u \in H^1_p(0,T; L_q(\R^n)) \cap L_p(0,T; H^2_q(\R^n \backslash \tilde \Sigma))$, cf. \cite{pruessbuch}.
Subtracting $\hat u$, we are left to solve
\begin{equation}  \label{3fffdfggdfggdfgj4}
\begin{cases} \begin{alignedat}{2}
\partial_t u + \omega u - \Delta u &= f_u, \qquad\qquad\qquad &\text{in } \R^n \backslash \tilde \Sigma, \\
\ljump u\rjump =0, \quad u|_{\tilde\Sigma} - \Delta h &= g, &\text{on } \tilde \Sigma, \\
\partial_t h + \omega h + \ljump \partial_3 u \rjump &= 0, &\text{on } \tilde \Sigma, \\
u(0) &= 0, &\text{in } \R^n \backslash \tilde\Sigma, \\
h(0) &= 0, &\text{on } \tilde\Sigma. \\
\end{alignedat} \end{cases}
\end{equation}
Let us now solve an auxiliary parabolic problem on the upper half space 
\begin{equation}  
\begin{cases} \begin{alignedat}{2}
\partial_t u_+ + \omega_+ u_+ - \Delta u_+ &= 0, \qquad\qquad\qquad &\text{in } \R^n_+, \\
 u_+|_{\tilde\Sigma}  &= g, &\text{on } \tilde \Sigma, \\
u_+(0) &= 0, &\text{in } \R^n_+, \\
\end{alignedat} \end{cases}
\end{equation}
by a function $u_+ \in H^1_p(0,T; L_q (\R^n_+)) \cap L_p(0,T; H^2_q(\R^n_+))$, cf. \cite{pruessbuch}. Note that the necessary compatibility condition $g(0) = 0$ is satisfied. Using higher-order reflection techniques we may extend the function $u_+$ in space to a function $\tilde u_+$ on all of $\R^n$ with the same regularity. Hence $\tilde u_+ (0) = 0$, $\ljump \tilde u_+ \rjump = 0$, and $\ljump \partial_3 \tilde u_+ \rjump = 0$. After subtracting $\tilde u_+$, we have arrived at the problem
\begin{equation}  \label{3ggdfgfgfghcvcvcvdffffdfcvcvj4}
\begin{cases} \begin{alignedat}{2}
\partial_t u + \omega u - \Delta u &= f_u, \qquad\qquad\qquad &\text{in } \R^n \backslash \tilde \Sigma, \\
\ljump u\rjump =0, \quad u|_{\tilde\Sigma} - \Delta h &= 0, &\text{on } \tilde \Sigma, \\
\partial_t h + \omega h + \ljump \partial_3 u \rjump &= 0, &\text{on } \tilde \Sigma, \\
u(0) &= 0, &\text{in } \R^n \backslash \tilde\Sigma, \\
h(0) &= 0, &\text{on } \tilde\Sigma. \\
\end{alignedat} \end{cases}
\end{equation}
We now want to solve \eqref{3ggdfgfgfghcvcvcvdffffdfcvcvj4} with properties of the Stefan semigroup, cf. \cite{pruessbuch}.
Corresponding to Section 6.6 in \cite{pruessbuch}, we define
\begin{equation}
X_0 := L_q(\R^n) \times W^{2-2/q}_q(\tilde \Sigma), \quad X_1 := H^2_q(\R^n \backslash \tilde \Sigma) \times W^{4-1/q}_q(\tilde \Sigma).
\end{equation}
Define the linear operator $A_S$ in $X_0$ by means of
\begin{equation}
A_S(u,h) = \begin{pmatrix}
-\Delta u \\
\ljump \partial_3 u \rjump 
\end{pmatrix},
\end{equation}
with domain
\begin{equation}
D(A_S) := \{ (u,h) \in X_1 : \ljump u \rjump = 0, \; u|_{\tilde\Sigma} - \Delta h = 0 \text{ on } \tilde\Sigma, \; \ljump \partial_3 u \rjump \in W^{2-2/q}_q(\tilde\Sigma) \}.
\end{equation}
Let $z := (u,h)$ and $f := (f_u,0)$. Then we may rewrite problem \eqref{3ggdfgfgfghcvcvcvdffffdfcvcvj4} as an abstract evolution equation in $X_0$, reading as
\begin{equation} \label{43534534534534534543534}
\frac{d}{dt} z(t) + \omega z(t) + A_S z(t) = f(t), \; t > 0, \quad z(0) = 0.
\end{equation}
By Section 6.6.3 in \cite{pruessbuch} the operator $(\omega + A_S)$ has maximal $L_q$-regularity. A general principle going back to Bourgain \cite{bourgain1984} then also gives maximal $L_p$-regularity for the abstract evolution problem \eqref{43534534534534534543534}. Therefore we may solve \eqref{43534534534534534543534} for $f := (f_u,0) \in L_p(0,T; X_0)$ to obtain a unique solution
\begin{equation}
z = (u,h) \in L_p(0,T; D(A_S)) \cap H^1_p(0,T; X_0)
\end{equation}
of \eqref{3ggdfgfgfghcvcvcvdffffdfcvcvj4}. Clearly, $u \in \mathbb E_u(T)$. Note that
\begin{equation}
\Delta h  = u|_{\tilde \Sigma} \in F^{1-1/2q}_{pq} (0,T; L_q (\tilde \Sigma)) \cap L_p(0,T; W^{2-1/q}_q (\tilde \Sigma) ).
\end{equation}
Hence we obtain $h \in F^{1-1/2q}_{pq} (0,T; H^2_q (\tilde \Sigma))$ by elliptic theory. Since also
\begin{equation}
\partial_t h = - \omega h - \ljump \partial_3 u \rjump \in F^{1/2 -1/2q}_{pq} (0,T; L_q(\tilde \Sigma)) \cap L_p(0,T; W^{2-2/q}_q(\tilde \Sigma)),
\end{equation}
we obtain $ h \in F^{3/2 -1/2q}_{pq} (0,T; L_q(\tilde \Sigma))$ and hence $h \in \mathbb E_h(T)$. In particular, also $h \in H^1_p(0,T; W^{2-2/q}_q(\tilde \Sigma))$. This also follows from the embedding $$\mathbb E_h(T) \into H^1_p(0,T; W^{2-2/q}_q(\tilde \Sigma)).$$
Let us summarize the result we have proven for this model problem.
\begin{theorem}[Maximal regularity] \label{34ztutzutzutzutzutzutzuf}
Let $n = 2,3$, $p \in (6,\infty)$, $q \in (19/10,2)\cap (2p/(p+1), 2)$, and $T \in (0,\infty)$.
Furthermore, let
$\Omega := \{ x \in \R^n : x_1 > 0 \}$, and $\Sigma := \Omega \cap \{ x_n = 0 \}.$ Then the linear problem with flat interface \eqref{43453454445444444454} has maximal $L_p-L_q$-regularity. More precisely, for every $(f_u,f_n,g,b,f_h,u_0,h_0)$ satisfying the regularity and compatibility conditions
\begin{enumerate}
\item $f_u \in L_p(0,T; L_q(\Omega))$,
\item $f_n \in F^{1/2-1/2q}_{pq} (0,T; L_q(\partial\Omega) ) \cap L_p(0,T; W^{1-1/q}_q(\partial \Omega))$,
\item $g \in F^{1-1/2q}_{pq} (0,T; L_q(\Sigma) ) \cap L_p(0,T; W^{2-1/q}_q(\Sigma))$,
\item $b \in {F}^{5/4 - 1/q + 1/4(3q-1)   }_{pq}(0,T; L_q (\partial \Sigma))  \cap {F}^{1-1/2q}_{pq}(0,T; W^{1-1/q}_q(\partial \Sigma))  \cap L_p(0,T ; W^{3-2/q}_q(\partial \Sigma))$, 
\item $f_h \in F^{1/2-1/2q}_{pq}(0,T; L_q(\Sigma)) \cap L_p(0,T; W^{1-1/q}_q(\Sigma))$, 
\item $u_0 \in W^{2-2/q}_q(\Omega \backslash \Sigma)$, 
\item $h_0 \in B^{4-1/q-2/p}_{qp}(\Sigma)$, 
\item $(n_{\partial\Omega} | \nabla u_0) = f_n(0)$, on $\partial\Omega$,
\item $\ljump u_0 \rjump = 0$, on $\Sigma$,
\item $u_0|_\Sigma - \Delta h_0 = g(0),$ on $\Sigma$,
\item $(n_{\partial\Sigma} | \nabla h_0) = b(0)$, on $\partial\Sigma$,
\item $\ljump \partial_3 u_0 \rjump - f_h(0) \in B^{2-2/q - 4/p}_{qp}(\Sigma)$,
\end{enumerate}
there exists a unique solution
\begin{gather}
u \in H^1_p(0,T; L_q(\Omega)) \cap L_p(0,T; H^2_q(\Omega \backslash \Sigma)), \\
h \in F^{3/2-1/2q}_{pq}(0,T; L_q (\Sigma))  \cap F^{1-1/2q}_{pq}(0,T; H^2_q(\Sigma))  \cap L_p(0,T ; W^{4-1/q}_q(\Sigma)),
\end{gather}
solving \eqref{43453454445444444454} on $(0,T)$.
Furthermore, the solution map $[(f_u,f_n,g,b,f_h,u_0,h_0) \mapsto (u,h)]$ is continuous with respect to these spaces. In particular, the operator norm of  $[(f_u,f_n,g,b,f_h,0,0) \mapsto (u,h)]$ is independent of $T > 0$.
\end{theorem}

Let us briefly comment on the second type of chart we obtain in a localization procedure for a cylindrical container: the quarter space problem with no interface. Here, $Q := \{ x \in \R^n : x_1 > 0, x_n < 0 \}$, where $n = 2,3$. The model problem reads as
\begin{equation} \label{434534ttttttttthhhh}
\begin{cases} \begin{alignedat}{2}
\partial_t u + \omega u - \Delta u &= f_u, \qquad\qquad\qquad &\text{in } Q, \\
\partial_1 u &= g_1, &\text{in } \partial Q \cap \{ x_1 = 0 \}, \\
\partial_n u &= g_n, &\text{in } \partial Q \cap \{ x_n = 0 \}, \\
u(0) &= u_0, &\text{in } Q. \\
\end{alignedat} \end{cases}
\end{equation}
Surely, one obtains compatibility conditions at $t = 0$, namely
\begin{equation}
g_1 (0 ) = \partial_1 u_0, \quad g_n(0) = \partial_n u_0.
\end{equation}
However, due to $q < 2$ there is no compatibility condition in space for the functions $g_1$ and $g_n$ on the triple line $\{ x_1 = 0, x_2 \in \R, x_3 = 0 \}$. This observation was already made in \cite{mulsekpaper12} in the stationary case. Following the lines of the arguments in Section A.2 in \cite{mulsekpaper12} we are able to reflect the problem to a half space problem which we can solve by classical results. For further discussion we refer to \cite{mulsekpaper12}. This way, we obtain a similar result on maximal regularity for the model problem \eqref{434534ttttttttthhhh} as Theorem \ref{34534534534545ff}.
\begin{theorem}[Maximal regularity] \label{34534534534545ff}
Let $n = 2,3$, $p \in (6,\infty)$, $q \in (19/10,2) \cap (2p/(p+1), 2)$, $T \in (0,\infty)$ and $Q$ the quarter-space as above. Let $\partial Q \cap \{ x_j = 0 \} =: \partial Q_j$.
Then for every $(g_1,g_n,u_0)$ satisfying the regularity and compatibility conditions
\begin{enumerate}
\item $f_u \in L_p(0,T; L_q(\Omega))$,
\item $g_j \in F^{1/2-1/2q}_{pq} (0,T; L_q(\partial Q_j ) ) \cap L_p(0,T; W^{1-1/q}_q(\partial Q_j)), \; j \in \{ 1,n \},$
\item $u_0 \in W^{2-2/q}_q(\Omega \backslash \Sigma)$, 
 \item $\partial_j u_0 = g_j(0), \; j \in \{1,n \}$,
\end{enumerate}
there exists a unique solution $(u,h)$ of \eqref{434534ttttttttthhhh} in the regularity classes of Theorem \ref{34534534534545ff}.
Furthermore, the solution map is continuous.
\end{theorem}

\subsection{Bent half-space problems and localization procedure}
We can extend the result to slightly bent half-spaces by a perturbation argument. This way we obtain a similar result as in the previous subsection if only the $C^1$-norm of the curves describing the bent half-spaces are small enough. Then afterwards a so-called localization procedure gives maximal regularity for the linear problem inside a bounded, cylindrical container. For a detailed derivation in case of a quasi-stationary Stefan problem we refer to Section 4.4 in \cite{mulsekpaper12}. We omit the details here.

\section{Nonlinear Well-Posedness}
In this section we show nonlinear well-posedness of the transformed version of the Stefan problem \eqref{34905384953745034875} given by \eqref{978987456885845fgdrerggregrrgrgg9685} in an $L_p-L_q$-setting. For convenience we recall that the transformed system is given by
\begin{equation}\label{345345345435345345}  \begin{cases}
\begin{alignedat}{2}
\partial_t u -  \Delta u  &= F_u (u,h), & \text{in } \Omega \backslash \Sigma, \\
\ljump u \rjump = 0, \quad  u|_\Sigma - \sigma \Delta_{x'} h &= F_\kappa (h), &\text{on } \Sigma, \\
(n_{\partial\Omega} | \nabla u)  &=0, \quad\quad\quad\quad& \text{on } \partial\Omega \backslash \partial\Sigma, \\ 
\partial_t h +  \ljump \partial_3 u \rjump &= F_\Sigma(u,h), \quad\quad\quad & \text{on } \Sigma, \\
( n_{\partial \Sigma}| \nabla h   ) &= 0, &\text{on } \partial \Sigma, \\
u(0) &= u_0, &\text{in } \Omega \backslash \Sigma, \\
h(0) &= h_0, &\text{on } \Sigma.
\end{alignedat} \end{cases}
\end{equation}
where 
\begin{equation}
\begin{alignedat}{1}
F_u (u,h) &:=  (\Delta_h - \Delta)u + Du \cdot \partial_t \Theta_h^{-1}, \\
F_\kappa(h) &:= \sigma \left[ \div \left( \frac{\nabla h}{ \sqrt{ 1 + |\nabla h|^2} } \right) - \Delta h \right], \\
F_\Sigma (u,h) &:= \ljump \partial_3 u - n_\Gamma \cdot \nabla_h u \rjump + \partial_t h (e_3| e_3 - n_\Gamma ).
\end{alignedat}
\end{equation}
Let us precisely state the function spaces we consider. Let
\begin{gather}
\mathbb E_u (T) := H^1_p(0,T; L_q(\Omega)) \cap L_p(0,T; H^2_q(\Omega \backslash \Sigma)), \\
\mathbb E_h (T) := F^{3/2-1/2q}_{pq}(0,T; L_q (\Sigma))  \cap F^{1-1/2q}_{pq}(0,T; H^2_q(\Sigma))  \cap L_p(0,T ; W^{4-1/q}_q(\Sigma)).
\end{gather}
The main result is the following.
\begin{theorem}
Let $p \in (6,\infty)$, $q \in (19/10, 2) \cap (2p/(p+1), 2)$. Then there is some $\delta_0 > 0$, such that if
\begin{equation}
|u_0 |_{W^{2-2/q}_q(\Omega \backslash \Sigma)} + |h_0|_{B^{4-1/q - 2/p}_{qp}(\Sigma)} \leq \delta
\end{equation}
for some $0 < \delta \leq \delta_0$,
and $(u_0,h_0)$ satisfy the compatibility conditions
\begin{enumerate}
\item $\ljump u_0 \rjump = 0$,
\item $ u_0|_\Sigma - \sigma \Delta h_0 = F_\kappa (h_0)$,
\item $(n_{\partial\Omega} | \nabla u_0) = 0$,  
\item $(n_{\partial \Sigma} | \nabla h_0) = 0$,
\item $\ljump \partial_3 u_0 \rjump - F_\Sigma (u_0,h_0) \in B^{2-2/q-4/p}_{qp}(\Sigma)$,
\end{enumerate}
there exists $\tau = \tau(\delta) >0$, such that the transformed Stefan problem \eqref{345345345435345345} has a unique strong solution $(u,h) \in \mathbb E_u(\tau) \times \mathbb E_h(\tau)$ on $(0, \tau)$.
\end{theorem}
\begin{proof}
The proof uses maximal regularity of the underlying linear problem together with a contraction argument via Banach's fixed point principle. It follows the same lines as the proof in \cite{mulsekpaper12}. We omit the details.
\end{proof}
\section{The Stefan semigroup} \label{jklughjkbgzhg345}
We now want to understand the structure of the underlying semigroup to the Stefan problem with ninety degree contact angle. We will extract the semigroup in the same way as is done in Section 6.6 in \cite{pruessbuch} in the case of closed interfaces.

Let $\Omega \subset \R^n$ be a bounded, smooth domain and $\Sigma$ a smooth surface inside with ninety degree contact angle as before.
Again we consider the linear problem
\begin{equation} \label{43453dfgdfgdfg454445444444dfgdfg454}
\begin{cases} \begin{alignedat}{2}
\partial_t u + \omega u - \Delta u &= f_u, \qquad\qquad\qquad &\text{in } \Omega \backslash \Sigma, \\
(n_{\partial\Omega} | \nabla u) &= f_n, &\text{in } \partial\Omega, \\
\ljump u\rjump =0, \quad u|_\Sigma - \Delta h &= g, &\text{on } \Sigma, \\
(n_{\partial\Omega} | \nabla h) &= b, &\text{in } \partial\Sigma, \\
\partial_t h + \omega h + \ljump \partial_3 u \rjump &= f_h, &\text{on } \Sigma, \\
u(0) &= u_0, &\text{in } \Omega \backslash \Sigma, \\
h(0) &= h_0, &\text{on } \Sigma. \\
\end{alignedat} \end{cases}
\end{equation}
Let $p \in (6,\infty)$, $q \in (19/10,2) \cap (2p/(p+1), 2p)$. We are now interested in strong solutions
\begin{equation} \label{43534534534534534534A}
\begin{alignedat}{2}
h \in \mathbb E_h^S (T) := H^1_p(0,T; W^{2-2/q}_q (\Sigma))  \cap F^{1-1/2q}_{pq}(0,T; H^2_q(\Sigma))  \cap L_p(0,T ; W^{4-1/q}_q(\Sigma))
\end{alignedat}
\end{equation}
and
\begin{equation}
\begin{alignedat}{2} \label{43534534534534534534B}
u \in \mathbb E_u^S (T) := & H^1_p(0,T; L_q(\Omega)) \cap L_p(0,T; H^2_q(\Omega \backslash \Sigma)) \\ &\quad \cap \{ u : \ljump \partial_ 3 u \rjump \in L_p(0,T; W^{2-2/q}_q(\Sigma)) \}.
\end{alignedat}
\end{equation}
Note that we now choose a different solution space for $h$ compared to the previous sections.
We obtain a different condition for $f_h$, namely $f_h \in L_p(0,T; W^{2-2/q}_q(\Sigma))$. 
Again we obtain compatibility conditions
\begin{enumerate}
\item $(n_{\partial\Omega} | \nabla u_0) = f_n(0)$, on $\partial\Omega$,
\item $\ljump u_0 \rjump = 0$, on $\Sigma$,
\item $u_0|_\Sigma - \Delta h_0 = g(0),$ on $\Sigma$,
\item $(n_{\partial\Sigma} | \nabla h_0) = b(0)$, on $\partial\Sigma$,
\item $\ljump \partial_3 u_0 \rjump \in \mathsf{tr}|_{t=0} \left[  F^{1/2 - 1/2q}_{pq} (0,T; L_q(\Sigma)) \cap L_p(0,T; W^{2-2/q}_q(\Sigma))  \right] $.
\end{enumerate}
 Let us explain the last one in more detail. Standard trace theory for $u \in \mathbb E_u^S(T)$ entails
\begin{equation}
\ljump \partial_3 u \rjump \in  F^{1/2-1/2q}_{pq} (0,T; L_q(\Sigma) ) \cap L_p(0,T; W^{1-1/q}_q(\Sigma)).
\end{equation}
Additionally, however, $\ljump \partial_3 u \rjump \in L_p(0,T; W^{2-2/q}_q(\Sigma))$ by \eqref{43534534534534534534B}. Hence the last compatibility condition is shown. 

We obtain maximal regularity also with respect to these function spaces. The proof is an adaption of the proof found in Section 6.6 in \cite{pruessbuch}. 

\begin{theorem}[Maximal regularity] \label{34ztutsdfddgzutzutzutzutzutzuf}
Let $n = 2,3$, $p \in (6,\infty)$, $q \in (19/10,2)\cap (2p/(p+1), 2)$ and $T \in (0,\infty)$.
Furthermore, let
$\Omega$ and $\Sigma$ be as above.
Then for every $(f_u,f_n,g,b,f_h,u_0,h_0)$ satisfying the regularity and compatibility conditions
\begin{enumerate}
\item $f_u \in L_p(0,T; L_q(\Omega))$,
\item $f_n \in F^{1/2-1/2q}_{pq} (0,T; L_q(\partial\Omega) ) \cap L_p(0,T; W^{1-1/q}_q(\partial \Omega))$,
\item $g \in F^{1-1/2q}_{pq} (0,T; L_q(\Sigma) ) \cap L_p(0,T; W^{2-1/q}_q(\Sigma))$,
\item $b \in {F}^{5/4 - 1/q + 1/4(3q-1)   }_{pq}(0,T; L_q (\partial \Sigma))  \cap {F}^{1-1/2q}_{pq}(0,T; W^{1-1/q}_q(\partial \Sigma))  \cap L_p(0,T ; W^{3-2/q}_q(\partial \Sigma))$, 
\item $f_h \in  L_p(0,T; W^{2-2/q}_q(\Sigma))$, 
\item $u_0 \in W^{2-2/q}_q(\Omega \backslash \Sigma)$, 
\item $h_0 \in B^{4-1/q-2/p}_{qp}(\Sigma)$, 
\item $(n_{\partial\Omega} | \nabla u_0) = f_n(0)$, on $\partial\Omega$,
\item $\ljump u_0 \rjump = 0$, on $\Sigma$,
\item $u_0|_\Sigma - \Delta h_0 = g(0),$ on $\Sigma$,
\item $(n_{\partial\Sigma} | \nabla h_0) = b(0)$, on $\partial\Sigma$,
\item $\ljump \partial_3 u_0 \rjump - f_h(0) \in B^{2-2/q - 4/p}_{qp}(\Sigma)$,
\end{enumerate}
there exists a unique solution $(u,h) \in \mathbb E_u^S (T) \times \mathbb E_h^S (T)$
solving \eqref{43453dfgdfgdfg454445444444dfgdfg454} on $(0,T)$.
Furthermore, the solution map $[(f_u,f_n,g,b,f_h,u_0,h_0) \mapsto (u,h)]$ is continuous with respect to these spaces. In particular, the operator norm of  $[(f_u,f_n,g,b,f_h,0,0) \mapsto (u,h)]$ is independent of $T > 0$.
\end{theorem}

\section{Convergence to equilibria}
In this section we investigate the long-time behaviour of solutions starting close to certain equilibria.

We will characterize the set of equilibria, study the spectrum of the linearization of the transformed Stefan problem \eqref{345345345435345345} around the trivial equilibrium $(u_*,h_*) = (0,0)$, and show that solutions starting sufficiently close to certain equilibria converge to equilibrium at an exponential rate in the interpolation norm.

Let us discuss equilibria. By testing the Stefan problem $\eqref{34905384953745034875}_1$ with its solution $u$ and recalling the transport identity $ \frac{d}{dt} \int_{\Gamma(t)} \sigma = - \int_{\Gamma(t)} \sigma H V$, cf. \cite{garcke232}, we derive 
\begin{equation}
\frac{d}{dt} \left[ \int_{\Gamma(t)} \sigma + \int_\Omega \frac{|u|^2}{2} \right] = - \int_\Omega |\nabla u|^2.
\end{equation}
A stationary solution $(u,\Gamma)$ satisfies that $u$ is constant. Hence also $H_\Gamma$ is constant and the set of equilibrium solutions is
\begin{equation}
\mathcal E = \{ (u, \Gamma) : H_\Gamma = const., \; u = \sigma H_\Gamma \}.
\end{equation}
Let us now additionally assume that $\Gamma$ is the graph of a function $h$ over $\Sigma$. In this case, we may even deduce that $H_\Gamma = 0$. Indeed, by shifting we may assume that $h$ is mean value free without changing $H_\Gamma$. A well-known formula for the mean curvature, an integration by parts, and the ninety-degree angle boundary condition give
\begin{equation}
0 = \int_\Sigma H_\Gamma h = - \int_\Sigma \frac{ | \nabla h|^2}{ \sqrt{1+ | \nabla h |^2}} .
\end{equation}
In particular, $h$ is constant and therefore $H_\Gamma = 0$. Consequently, also $u = 0$.

We will now study the problem for the height function \eqref{345345345435345345} in an $L_p$-setting. The equilibria in the graph case are now
\begin{equation}
\mathcal E_\Sigma = \{ (u, h) : u = 0, \; h = const. \}.
\end{equation}
The linearization of the (transformed) Stefan problem with Gibbs-Thomson correction around the trivial equilibrium reads as
\begin{equation} \label{43dfgdfghjghjghjghjfffdfg4}
\begin{cases} \begin{alignedat}{2}
\partial_t u  - \Delta u &= f_u, \qquad\qquad\qquad &\text{in } \Omega \backslash \Sigma, \\
(n_{\partial\Omega} | \nabla u) &= f_n, &\text{in } \partial\Omega, \\
\ljump u\rjump =0, \quad u|_\Sigma - \Delta h &= g, &\text{on } \Sigma, \\
(n_{\partial\Sigma} | \nabla h) &= 0, &\text{in } \partial\Sigma, \\
\partial_t h + \ljump \partial_3 u \rjump &= f_h, &\text{on } \Sigma, \\
u(0) &= u_0, &\text{in } \Omega \backslash \Sigma, \\
h(0) &= h_0, &\text{on } \Sigma. \\
\end{alignedat} \end{cases}
\end{equation}
Assuming $f_n = g = 0$, we may rewrite \eqref{43dfgdfghjghjghjghjfffdfg4} as an abstract evolution equation as follows. We note that we have to choose the function spaces corresponding to the semigroup approach, cf. Section \ref{jklughjkbgzhg345}. Define Banach spaces
\begin{equation}
X_0 := L_q (\Omega) \times W^{2-2/q}_q(\Sigma), \quad X_1 := H^2_q(\Omega \backslash \Sigma) \times W^{4-1/q}_q(\Sigma),
\end{equation}
and the linear operator $A$ in $X_0$ by $A : D(A) \subset X_1 \pfeil X_0$,
\begin{equation}
A(u,h) := (-\Delta u, \ljump \partial_3 u \rjump ),
\end{equation}
with domain
\begin{align}
D(A) := \{ (u,h) \in X_1 : &\ljump u \rjump =0, \; u|_\Sigma = \Delta h, \; \ljump \partial_3 u \rjump \in W^{2-2/q}_q(\Sigma), \\ &(\nabla h, n_{\partial\Sigma}) = 0, \; (n_{\partial\Omega} | \nabla u) = 0 \}.
\end{align}
For $f_u \in L_p(0,T; L_q (\Omega))$, $f_h \in L_p(0,T; W^{2-2/q}_q(\Sigma))$, and $f_n = g = 0$, we may rewrite \eqref{43dfgdfghjghjghjghjfffdfg4} as an abstract evolution equation
\begin{equation}
\dot z (t) + Az(t) = f(t), \; t > 0, \quad z(0) = z_0,
\end{equation}
where $f := (f_u, f_h) \in L_p(0,T; X_0)$ and $z_0 := (u_0, h_0)$. The operator $A$ has the following properties.
\begin{lemma} \label{3445kk34k5jhl345}
Let $n = 2,3$, $p \in (6,\infty)$, $q \in (19/10, 2) \cap (2p/(p+1), 2)$, $A$ and $X_0$ as above.
\begin{enumerate}
\item The linear operator $-A$ generates an analytic $C_0$-semigroup $e^{-At}$ in $X_0$, which has maximal $L_p$-regularity.
\item The spectrum $\sigma(-A)$ consists of at most countably many eigenvalues with finite algebraic multiplicity.
 \item $\sigma(-A) \cap i\R \subset \{ 0 \}$.
\item $\sigma(-A) \backslash \{ 0 \} \subset \mathbb C_- := \{ z \in \mathbb C : \operatorname{Re} z < 0 \}$. 
\item $\lambda = 0$ is semi-simple with multiplicity one, $X_0 = N(A) \oplus R(A)$. 
\item $N(A)$ is isomorphic to the tangent space $T_{z_*} \mathcal E_\Sigma$ at the trivial equilibrium $z_* = (0,0)$. Furthermore, $N(A)$ is one dimensional and spanned by $(0,1)$.
\item The restriction of the semigroup $e^{-At}$ to ${R(A)}$ is exponentially stable.
\end{enumerate}
\end{lemma}
\begin{proof}
The first statement follows from the semigroup approach of Section \ref{jklughjkbgzhg345}. Since the domain $D(A)$ compactly embeds into $X_0$, the resolvent of $A$ is compact and hence the second statement follows. Let $\lambda \in \sigma(-A)$ with eigenfunctions $(u,h)$. The corresponding eigenvalue problem then reads as
\begin{equation} \label{4hhjhjhtzjtzjtzjfdfg4}
\begin{cases} \begin{alignedat}{2}
\lambda u  - \Delta u &= 0, \qquad\qquad\qquad &\text{in } \Omega \backslash \Sigma, \\
(n_{\partial\Omega} | \nabla u) &= 0, &\text{in } \partial\Omega, \\
\ljump u\rjump =0, \quad u|_\Sigma - \Delta h &= 0, &\text{on } \Sigma, \\
(n_{\partial\Sigma} | \nabla h) &= 0, &\text{in } \partial\Sigma, \\
\lambda h + \ljump \partial_3 u \rjump &= 0, &\text{on } \Sigma.
\end{alignedat} \end{cases}
\end{equation}
By testing $\eqref{4hhjhjhtzjtzjtzjfdfg4}_1$ with $u$ in $L_2(\Omega)$ and invoking the other equations an integration by parts entails
\begin{equation} \label{3456kkkkkk4564kkk456}
\lambda |u|_{L_2(\Omega)}^2 + |Du|_{L_2(\Omega)}^2 + \bar \lambda | \nabla h |_{L_2(\Sigma)}^2 = 0.
\end{equation}
This shows $\operatorname{Re} \lambda \leq 0$. If $\lambda = 0$, we obtain that $Du = 0$ and $u$ is constant in $\Omega$. Hence $\Delta h$ is constant on $\Sigma$. Integrating $\Delta h$ over $\Sigma$ and invoking $\eqref{4hhjhjhtzjtzjtzjfdfg4}_4$ gives $\Delta h = 0$. Hence $h = const.$ This shows that $\lambda = 0$ is an eigenvalue and that the kernel $N(A)$ is spanned by $(0,1)$.

Let us show that $\lambda = 0$ is the only eigenvalue with real part zero. Taking real parts in \eqref{3456kkkkkk4564kkk456} and using that $\operatorname{Re} \lambda = 0$ gives that
the corresponding eigenfunction $u$ is constant. Equation $\eqref{4hhjhjhtzjtzjtzjfdfg4}_5$ entails $\lambda h = 0$, and since $h$ may not be trivial, $\lambda = 0$.
We now prove $N(A) = N(A^2)$. Let $(u,h) \in N(A^2)$ and $(u_1, h_1) := A(u,h)$. Since now $(u_1,h_1)$ is in the kernel of $A$, $u_1 = 0$ and $h_1 = const.$ The problem for $(u,h)$ reads as
\begin{equation} \label{4hfgfdgfdgfdgfghfghfghfghdfgdfgdfgfg4}
\begin{cases} \begin{alignedat}{2}
   \Delta u &= 0, \qquad\qquad\qquad &\text{in } \Omega \backslash \Sigma, \\
(n_{\partial\Omega} | \nabla u) &= 0, &\text{in } \partial\Omega, \\
\ljump u\rjump =0, \quad u|_\Sigma - \Delta h &= 0, &\text{on } \Sigma, \\
(n_{\partial\Sigma} | \nabla h) &= 0, &\text{in } \partial\Sigma, \\
  \ljump \partial_3 u \rjump &= h_1, &\text{on } \Sigma.
\end{alignedat} \end{cases}
\end{equation}
An integration by parts on $\Omega^\pm$ using $\eqref{4hfgfdgfdgfdgfghfghfghfghdfgdfgdfgfg4}_{1,2}$ entails that
\begin{equation}
\int_\Sigma \ljump \partial _ 3 u \rjump dx'= \int_{\Omega^+} \div \nabla u dx + \int_{\Omega^-} \div \nabla u dx = 0.
\end{equation}
Since $h_1$ is constant, $h_1 = 0$ by \eqref{4hfgfdgfdgfdgfghfghfghfghdfgdfgdfgfg4}.
Consequently, $(u,h) \in N(A)$. This shows $N(A) = N(A^2)$, whence by spectral theory the range of $A$ is closed in $X_0$ and there is a spectral decomposition $X_0 = R(A) \oplus N(A)$, cf. \cite{engelnagel}, \cite{lunardioptimal}. Hence $\lambda = 0$ is semi-simple. In particular, the restricted semigroup $e^{-At}|_{R(A)}$ is exponentially stable since we have a spectral gap.
Note that we have even shown that in fact $N(A) = \mathcal E_\Sigma = T_{z_*} \mathcal E_\Sigma = span(0,1)$.
\end{proof}

\subsection*{Parametrization of nonlinear phase manifold.}
Let us parametrize the nonlinear phase manifold 
\begin{equation}
\begin{alignedat}{2}
\mathsf{PM} := \{ (u,h) &\in W^{2-2/q}_q(\Omega\backslash \Sigma) \times B_{qp}^{4-1/q-2/p}(\Sigma) : \ljump u \rjump = 0 \text{ on } \Sigma, \\
&u|_\Sigma - \sigma \Delta h  = F_\kappa (h)\text{ on } \Sigma, \; (n_{\partial\Omega} | \nabla u) = 0\text{ on } \partial\Omega, \\ & (n_{\partial\Sigma} | \nabla h) = 0 \text{ on } \partial\Sigma \}.
\end{alignedat}
\end{equation}
as a subset of $X_\gamma := W^{2-2/q}_q(\Omega\backslash \Sigma) \times B_{qp}^{4-1/q-2/p}(\Sigma)$ over 
\begin{equation}
\begin{alignedat}{2}
\mathsf{PM}_0 := \{ (u,h) &\in W^{2-2/q}_q(\Omega\backslash \Sigma) \times B_{qp}^{4-1/q-2/p}(\Sigma) : \ljump u \rjump = 0 \text{ on } \Sigma, \\
&u|_\Sigma - \sigma \Delta h  = 0\text{ on } \Sigma, \; (n_{\partial\Omega} | \nabla u) = 0\text{ on } \partial\Omega, \\ & (n_{\partial\Sigma} | \nabla h) = 0 \text{ on } \partial\Sigma \},
\end{alignedat}
\end{equation}
at least locally around the trivial equilibrium $(u_*, h_*) = (0,0)$. We note that $F_\kappa$ is smooth, $F_\kappa (0) = 0$, and $DF_\kappa(0) = 0$ since $F_\kappa$ is quadratic in $h$. Let us consider the stationary auxiliary problem
\begin{equation} \label{434fgdfdfgdfgdfggff4}
\begin{cases} \begin{alignedat}{2}
 \omega u - \Delta u &= 0, \qquad\qquad &\text{in } \Omega \backslash \Sigma, \\
(n_{\partial\Omega} | \nabla u) &= 0, &\text{in } \partial\Omega, \\
\ljump u\rjump =0, \quad u|_\Sigma &= g, &\text{on } \Sigma.
\end{alignedat} \end{cases}
\end{equation}
By a localization argument we may solve \eqref{434fgdfdfgdfgdfggff4} as follows, cf. also Appendix A in \cite{mulsekpaper12}.
\begin{lemma}
Let $n = 3$, $3/2 < q < 2$. For sufficiently large $\omega > 0$ we may solve \eqref{434fgdfdfgdfgdfggff4} for given right hand side $g \in W^{2-3/q}_q(\Sigma)$ uniquely by some $u \in W^{2-2/q}_q(\Omega \backslash \Sigma)$. Moreover, $[W^{2-3/q}_q(\Sigma) \ni g \mapsto u \in W^{2-2/q}_q(\Omega \backslash \Sigma)]$ is continuous.
\end{lemma}
We can now parametrize $\mathsf{PM}$ over $\mathsf{PM}_0$ as follows. By the nature of these manifolds, we can fix $h$ and only parametrize over $u$. Fix $\omega > 0$ large enough and denote by $L$ the linear operator on the left hand side of \eqref{434fgdfdfgdfgdfggff4}. Given a function $u$ in the linear phase manifold, we have that $ u = L^{-1} \sigma \Delta h$ for the fixed $h$. The function $ \hat u :=  u + L^{-1} F_\kappa (h) = L^{-1}[\sigma \Delta h + F_\kappa(h)]$ then lies in $\mathsf {PM}$ by construction. Note that the mapping $\psi : W^{2-3/q}_q(\Sigma) \pfeil W^{2-3/q}_q(\Sigma)$ given by $[ u \mapsto \hat u]$ is locally invertible around zero. Indeed, $D\psi(0) = I + L^{-1} DF_\kappa(0) = I$. Hence there is a small neighbourhood $U \subset W^{2-3/ q}_q(\Sigma)$ of zero, such that $[u \mapsto \psi(u)]$ is a parametrization of $\mathsf{PM}$ over $\mathsf{PM}_0$. Note that also $\psi \in C^\infty$ since $F_\kappa$ is smooth.
\subsection*{Convergence to equilibrium solutions.}
We can now formulate and show the main result on convergence to equilibrium solutions of this chapter.
\begin{theorem} \label{3948560384765083746534}
The trivial equilibrium $(0,\Sigma)$ is stable in the following sense. For each $\epsilon > 0$ there is some $\delta = \delta (\epsilon) >0$, such that for all initial values $(u_0,h_0) \in \mathsf{PM}$ subject to the smallness condition
\begin{equation}
|u_0|_{W^{2-2/q}_q(\Omega \backslash \Sigma)} + |h_0|_{B^{4-1/q-2/p}_{qp}(\Sigma)} \leq \delta(\epsilon)
\end{equation}
there exists a unique global in time solution $(u,h)$ of the transformed Stefan problem \eqref{345345345435345345} and it satisfies
\begin{equation}
|u(t)|_{W^{2-2/q}_q(\Omega \backslash \Sigma)} + |h(t)|_{B^{4-1/q-2/p}_{qp}(\Sigma)} \leq \epsilon, \quad t \in \R_+.
\end{equation}
Moreover there exists some constant $h_\infty$ such that
\begin{equation}
\left[ |u(t)|_{W^{2-2/q}_q(\Omega \backslash \Sigma)} + |h(t)-h_\infty|_{B^{4-1/q-2/p}_{qp}(\Sigma)} \right] \pfeil 0, \quad t \pfeil \infty.
\end{equation}
The convergence is at an exponential rate.
\end{theorem}
\begin{proof}
The proof is a modification to the proofs in \cite{abelswilke} and \cite{wilkehabil}. We shall only give the necessary modifications to it. Pick some $(u_0,h_0) \in \mathsf{PM}$ for some $\delta > 0$ to be chosen later. With the help of the parametrization of the phase manifold we can decompose $(u_0,h_0) = (\tilde u_0, \tilde h_0) + (\phi(\tilde u_0, \tilde h_0), 0)$, where $\tilde h_0 = h_0$ and $ (\tilde u_0, \tilde h_0) $ belongs to $\mathsf{PM}_0$. We have seen that actually $\phi(\tilde u_0, \tilde h_0) = \psi(\tilde u_0)$.

We now want to decompose the solution $(u,h) = (u_\infty, h_\infty) + (\tilde u, \tilde h) + (\bar u, \bar h)$, where $(u_\infty, h_\infty)$ is an equilibrium solution and $(\tilde u, \tilde h)(t)$ belongs to the linear phase manifold for each $t$. Note that $u_\infty = 0$ and $h_\infty$ is constant. Let $\omega > 0$ and consider the two coupled systems
\begin{equation}\label{345dfggbdfgdfg345}  \begin{cases}
\begin{alignedat}{2}
\omega \bar u +\partial_t \bar u -  \Delta \bar u  &= F_u (\tilde u + \bar u, h_\infty + \tilde h + \bar h), & \text{in } \Omega \backslash \Sigma, \\
\ljump \bar u \rjump = 0, \quad  \bar u|_\Sigma - \sigma \Delta \bar h &= F_\kappa ( h_\infty + \tilde h + \bar h), &\text{on } \Sigma, \\
(n_{\partial\Omega} | \nabla \bar u)  &=0, \quad\quad\quad\quad& \text{on } \partial\Omega \backslash \partial\Sigma, \\ 
\omega \bar h + \partial_t \bar h +  \ljump \partial_3 \bar u \rjump &= F_\Sigma(\tilde u + \bar u, h_\infty + \tilde h + \bar h), \quad\quad\quad & \text{on } \Sigma, \\
( n_{\partial \Sigma}| \nabla \bar h   ) &= 0, &\text{on } \partial \Sigma, \\
\bar u(0) &= \phi(\tilde u_0), &\text{in } \Omega \backslash \Sigma, \\
\bar h(0) &= 0, &\text{on } \Sigma.
\end{alignedat} \end{cases}
\end{equation}
and
\begin{equation}\label{3gdfgdfgeghjghjghj345}  \begin{cases}
\begin{alignedat}{2}
\partial_t \tilde u -  \Delta \tilde u   &= \omega \bar u, & \text{in } \Omega \backslash \Sigma, \\
\ljump \tilde u \rjump = 0, \quad  \tilde u|_\Sigma - \sigma \Delta \tilde h &= 0, &\text{on } \Sigma, \\
(n_{\partial\Omega} | \nabla \tilde u)  &=0, \quad\quad\quad\quad& \text{on } \partial\Omega \backslash \partial\Sigma, \\ 
\partial_t \tilde h +  \ljump \partial_3 \tilde u \rjump &= \omega \bar h, \quad\quad\quad & \text{on } \Sigma, \\
( n_{\partial \Sigma}| \nabla \tilde h   ) &= 0, &\text{on } \partial \Sigma, \\
u(0) &= \tilde u_0, &\text{in } \Omega \backslash \Sigma, \\
h(0) &= h_0 - h_\infty, &\text{on } \Sigma.
\end{alignedat} \end{cases}
\end{equation}
We start with \eqref{3gdfgdfgeghjghjghj345}. Note that the right hand side $\omega \bar h$ has high regularity properties, so the semigroup theory of the Stefan problem is applicable. With the help of the operator $A$ we may rewrite \eqref{3gdfgdfgeghjghjghj345} as an abstract evolutionary problem
\begin{equation} \label{klhj43j5hjjh345345354}
\frac{d}{dt} \tilde z(t) + A \tilde z(t) = R( \bar z)(t), \; t >0, \quad \tilde z(0) = \tilde z_0 - z_\infty,
\end{equation}
where $\tilde z = (\tilde u, \tilde h)$, $\bar z = (\bar u, \bar h)$, $\tilde z_0 = (\tilde u_0, \tilde h_0)$, $z_\infty = (0,h_\infty)$ and $R(\bar z) = \omega \bar z(t)$.

Thanks to Lemma \ref{3445kk34k5jhl345} we have that $A$ has maximal $L_p$-regularity on finite time intervals in the base space $$X_0= L_q (\Omega) \times W^{2-2/q}_q(\Sigma)$$ and $X_0  = R(A) \oplus N(A)$. Denote by $P^c$ the spectral projection corresponding to $\sigma_c (A) = \{ 0 \}$ and let $P^s := I - P^c$. Then $R(P^c) = N(A)$ and $R(P^s) = R( A)$.
Note that since $N(A) = \mathcal E_\Sigma$ in a neighbourhood of zero, we can parametrize $N(A)$ over $\mathcal E_\Sigma$ locally around zero via the identity map.
Define $A_s := A P^s$ and introduce new variables $\mathsf x := P^c \tilde z$ and $\mathsf y :=P^s \tilde z$. We obtain the so called normal form of \eqref{klhj43j5hjjh345345354},
\begin{equation} \begin{alignedat}{2} \label{34kijl5645k6j45j64kj5l64kj56l4k56}
\dot {\mathsf x }&= T(\bar z), \quad\quad &&\mathsf x(0) = \mathsf x_0 - \mathsf x_\infty, \\
\dot {\mathsf y }+ A_s \mathsf y &= S(\bar z), &&\mathsf y (0) = \mathsf y_0,
\end{alignedat}
\end{equation}
where $T(\bar z) = P^c R(\bar z)$, $S(\bar z) = P^s R(\bar z)$ and $\mathsf x_0 := P^c \tilde z_0$, $\mathsf y_0 := P^s \tilde z_0$. We note that in contrast to \cite{abelswilke} the equations for $\mathsf x$ and $\mathsf y$ decouple. We define $\mathsf x_\infty := \mathsf x_0 + \int_0^\infty T(\bar z) (s) ds$ and directly solve the first equation by
\begin{equation}
\mathsf x (t ) := - \int_t^\infty T(\bar z ) (s) ds.
\end{equation}
Note in particular that in our case, $T(\bar z) = \omega P^c (\bar u, \bar h) = (0, \omega \Pi \bar h )$, where $\Pi \bar h$ denotes the mean value of $\bar h$, $\Pi \bar h =(\bar h|1)_{L_2(\Sigma)}$.
Conclusively, $T(\bar z)$ is constant in space for each $t$.
Define for $\delta > 0$ the time-weighted spaces
\begin{align}
\mathbb E_u (\delta ; \R_+) &:= \{ u \in L_p(\R_+ ; L_q(\Omega) ) : e^{\delta t}u \in \mathbb E_u (\R_+) \}, \\
\mathbb E_h (\delta ; \R_+) &:= \{ h \in L_p(\R_+ ; L_q(\Sigma) ) : e^{\delta t}h \in \mathbb E_h (\R_+) \}.
\end{align}
where $\mathbb E_u (\R_+) := H^1_p(\R_+; L_q(\Omega)) \cap L_p(\R_+; H^2_q(\Omega \backslash \Sigma))$ and
\begin{equation}
\mathbb E_h (\R_+) := F^{3/2-1/2q}_{pq}(\R_+; L_q (\Sigma))  \cap F^{1-1/2q}_{pq}(\R_+; H^2_q(\Sigma))  \cap L_p(\R_+; W^{4-1/q}_q(\Sigma)).
\end{equation}
For given $(\bar u, \bar h) \in \mathbb E_u (\delta ; \R_+) \times \mathbb E_h (\delta ; \R_+)$ we see that $T(\bar z) = (0,T_2(\bar z)) $, where $T_2(\bar z) \in \mathbb E_h (\delta ; \R_+)$. Now,
\begin{equation}
e^{\delta t} \mathsf x (t) = - \int_t^\infty e^{\delta (t-s)} e^{\delta s} T(\bar z)(s) ds,
\end{equation}
whence Young's inequality gives $e^{\delta t} \mathsf x \in L_p(\R_+  ; C^4(\Sigma) )$ since also $e^{\delta t} \mathsf x$ is constant in space. Now,
 \begin{equation}
 \frac{d}{dt}(e^{\delta t} \mathsf x) = \delta e^{\delta t} \mathsf x + e^{\delta t} \dot{ \mathsf x}, \quad  \frac{d^2}{dt^2} (e^{\delta t} \mathsf x) = \delta^2 e^{\delta t} \mathsf x + 2\delta e^{\delta t} \dot{\mathsf x} + \delta e^{\delta t} { \ddot { \mathsf x}}.
  \end{equation}
 Furthermore, $\dot{ \mathsf x} (t) = T(\bar z)(t)$ and $\ddot{ \mathsf x} (t) = \frac{d}{dt}T(\bar z)(t)$. Since $T(\bar z) \in H^1_p(\R_+ , \delta ; C^4(\Sigma))$ we obtain that $e^{\delta t} \mathsf x \in H^2_p(\R_+ , \delta ; C^4(\Sigma)) \into \mathbb E_h (\R_+)$.
Let us solve $\eqref{34kijl5645k6j45j64kj5l64kj56l4k56}_2$. Since we have a spectral gap, the operator $A_s$ has maximal $L_p$-regularity on $X_0$ on the half line $\R_+$. Recall that $S(\bar z) =\omega P^s (\bar u, (I-\Pi) \bar h)$. For given $(\bar u, \bar h) \in \mathbb E_u (\delta ; \R_+) \times \mathbb E_h (\delta ; \R_+)$ we now note that
\begin{equation}
\bar h \in H^1_p(\R_+ , \delta ; W^{2-2/q}_q(\Sigma)) \cap L_p(\R_+ , \delta ; W^{4-1/q}_q(\Sigma))
\end{equation}
by embedding. If now $0 < \delta \leq \delta_0$ for some sufficiently small $\delta_0 > 0$ depending on the spectral bound of $A_s$ we may solve $\eqref{34kijl5645k6j45j64kj5l64kj56l4k56}_2$ in exponentially time-weighted spaces to the result 
\begin{equation}
\mathsf y \in \mathbb E_u(\delta ; \R_+) \times [ H^1_p(\R_+ , \delta ; W^{2-2/q}_q(\Sigma)) \cap L_p(\R_+ , \delta ; W^{4-1/q}_q(\Sigma))  ].
\end{equation}
It is now noteworthy that $\mathsf y_2$ posseses some more regularity by the nature of the equations,
\begin{equation}
 \partial_t \mathsf y_2  = - \ljump \partial_3 \mathsf y_1 \rjump + \omega (I- \Pi)\bar h  \in e^{-\delta t} F^{1/2-1/2q}_{pq} (\R_+ ; L_q(\Sigma)).
\end{equation}
Putting things together we see that $\tilde z = \mathsf x + \mathsf y$ and $z_\infty = \mathsf x_\infty$.

We now consider \eqref{345dfggbdfgdfg345}. Let $L_\omega$ be the linear operator defined by the left hand side of \eqref{345dfggbdfgdfg345}. We may rewrite the problem then abstractly as
\begin{equation} 
L_\omega \bar z = N(z_\infty + \tilde z + \bar z), \; t > 0, \quad \bar z(0) = \bar z_0 := ( \phi(\tilde u_0, \tilde h_0) , 0).
\end{equation}
By the above considerations there exists a function $\tilde H$ such that $\tilde z = \tilde H ( \mathsf x_0, \mathsf y_0, \bar z)$. In order to resolve the compatibility conditions and to solve \eqref{345dfggbdfgdfg345} we define
\begin{equation} \label{34k5hll3jh4534}
M(\mathsf x_0, \mathsf y_0, \bar z ) := N (z_\infty + \tilde z + \operatorname{ext}_\delta [ (\phi(\tilde u_0, \tilde h_0) , 0) - (\bar u (0), \bar h (0)) ] + \bar z ),
\end{equation}
where $\operatorname{ext}_\delta : W^{2-2/q}_q(\Omega \backslash \Sigma) \times B^{4-1/q-2/p}_{qp}(\Sigma) \pfeil \mathbb E_u (\delta ; \R_+ ) \times E_h (\delta ; \R_+ )$ is an extension operator satisfying $(\operatorname{ext}_\delta z ) (0) = z$. We may now argue as in \cite{abelswilke} and solve
\begin{equation} \label{435kljh345hk345hk34}
L_\omega \bar z = M(\mathsf x_0, \mathsf y_0, \bar z ), \; t >0, \quad \bar z(0) = ( \phi(\tilde u_0, \tilde h_0), 0),
\end{equation}
by the implicit function theorem locally around zero. Indeed, we also have that $L_\omega$ is invertible, provided $\omega > 0$ is large enough. To see this we shall consider the linear system $L_\omega u = F$, for a given right hand side $F$. By the reduction procedure in the proof of Theorem \ref{34ztutzutzutzutzutzutzuf} we have shown that without loss of generality we may assume that $F=(f_u,0,0,0,0,0,0)$. We may then write the problem abstractly as an evolution equation
\begin{equation}
\omega u + \dot u + A u = (f_u, 0), \; t > 0, \quad u(0) = 0,
\end{equation}
where $A$ is as before and has maximal regularity of type $L_p$ on bounded intervals. By now choosing $\omega >0 $ large enough, we obtain that the spectral bound of $\omega + A$ is strictly negative, hence $\omega + A$ has maximal regularity on the half line. Again we note that $u_2$ then enjoys some additional time regularity, also in exponentially time-weighted spaces. 

Define now the map
\begin{equation}
K( \mathsf x_0, \mathsf y_0, \bar z ) := \bar z - (L_\omega, \operatorname{tr}_{t=0} )^{-1} ( M (\mathsf x_0, \mathsf y_0, \bar z), ( \phi(\tilde u_0, \tilde h_0), 0 ) ).
\end{equation}
By construction in \eqref{34k5hll3jh4534}, $M$ resolves the compatibility conditions at $t = 0$. Furthermore, note that $M(0) = M' (0) = 0$, whence as in \cite{abelswilke} we may apply the implicit function theorem to obtain a solution $\bar z$ of \eqref{345dfggbdfgdfg345}.

Putting things together, we see that $(\tilde u (t), \tilde h(t) )$ as well as $(\bar u (t), \bar h(t))$ converge to zero in $X_\gamma$ as $t \pfeil 0$ at an exponential rate.
\end{proof}

\section{Two-phase Navier-Stokes/Stefan problems}
In this section we want to couple the Stefan problem with Gibbs-Thomson correction and the ninety degree contact angle condition to the two-phase Navier-Stokes equations with surface tension in a capillary domain. This model without boundary contact has already been studied widely as a model for incompressible two-phase fluid flows with phase transitions, cf. \cite{pruessbuch}. It is well-known in the case where no boundary contact occurs that the Navier-Stokes part is only weakly coupled to the Stefan problem, cf. Remark 1.3.3 in \cite{pruessbuch}. We will show rigorously that this is also the case when a ninety degree contact angle problem is present. 

We shall now precisely formulate the model we consider. For simplicity let $n=3$. Let $-\infty < L_1 < 0 < L_2 < \infty$,
$\Omega := \Sigma \times (L_1,L_2) \subset \R^3$, and $\Sigma \subset \R^2$ bounded and smooth. Let again $S_1 := \partial \Sigma \times (L_1,L_2)$ and $S_2 := \Sigma \times \{ L_1, L_2 \}$.

We consider the coupled system
\begin{equation} \label{34ghfghfgh75} \begin{cases}
\begin{alignedat}{2}
\partial_t v + (v \cdot \nabla) v - \mu^\pm \Delta v + \nabla p &= 0, &\text{in } \Omega \backslash \Gamma(t), \\
\div v &= 0, &\text{in } \Omega \backslash \Gamma(t), \\
- \ljump \mu^\pm (Dv + Dv^\top) - pI \rjump \nu_\Gamma &= \sigma H_\Gamma \nu_\Gamma, &\text{on } \Gamma(t), \\
\ljump v \rjump &= 0, &\text{on }  \Gamma(t), \\
\partial_t u - \Delta u &= 0, &\text{in } \Omega \backslash \Gamma(t), \\
 \ljump u \rjump =0, \quad u|_\Gamma &= \sigma H_\Gamma &\text{on } \Gamma(t), \\
n_{\partial \Omega} \cdot \nabla u  &= 0,&\text{on } \partial\Omega \backslash \partial\Gamma(t), \\
V_\Gamma &= v|_\Gamma \cdot \nu_\Gamma - \ljump \nu_\Gamma \cdot \nabla u \rjump,  \quad &\text{on }  \Gamma(t), \\
P_{S_1} \left( \mu^\pm (Dv + Dv^\top) \nu_{S_1} \right) &= 0, & \text{on } S_1 \backslash \partial\Gamma(t), \\
v \cdot \nu_{S_1} &= 0, & \text {on } S_1 \backslash \partial \Gamma(t), \\
v|_{S_2} &= 0, & \text{on } S_2, \\
\nu_\Gamma \cdot \nu_{\partial\Omega}  &= 0, &\text{on } \partial \Gamma(t), \\
v|_{t=0} &= v_0, & \text{in } \Omega \backslash \Gamma(0), \\
u|_{t=0} &= u_0, &\text{in } \Omega \backslash \Gamma(0), \\
\Gamma|_{t=0} &= \Gamma_0.
\end{alignedat} \end{cases}
\end{equation}

Let us briefly comment on the pure-slip boundary condition for $v$. We want to reflect the problem at the contact line across the boundary of the domain to draw back the problem to a two-phase full space problem. Here, Navier-conditions seem to be the canonical choice in the literature, cf. \cite{wilkehabil}. Since we also need the two boundary conditions at $S_1$ and $S_2$ to be compatible on $\overline S_1 \cap \overline S_2$, we formulate pure-slip conditions. For further discussion we refer to \cite{navmulsekpap} and \cite{wilkehabil}. 

\subsection*{Maximal regularity for linear problem.}
 The main ingredient to the proof of maximal regularity is that the Navier-Stokes problem is only weakly coupled. 
 Note that we treated the Stefan problem in an $L_p-L_q$-theory with $p \not= q$. To obtain better regularity than $L_q$ for the velocity solution $v$, we introduce a third integration scale $L_r$, where $q < r < p$. This may seem at first unnecessarily technical, however we circumvent certain problems with this ansatz while still getting good regularity for $v$. For a detailed discussion we refer to Section 4.2 in \cite{navmulsekpap}.

 \begin{lemma} \label{kl3j4h5jh345345}
Let $p \in (6,\infty)$, $q \in (2p/(p+1),2) \cap (19/10,2)$, and $0 < T \leq T_0$ for some fixed $T_0 < \infty.$ 
Then there is some $r \in (5, 6)$, such that
for any $h \in F^{3/2-1/2q}_{pq}(0,T; L_q (\Sigma))  \cap F^{1-1/2q}_{pq}(0,T; H^2_q(\Sigma))  \cap L_p(0,T ; W^{4-1/q}_q(\Sigma))$, we have that
\begin{equation} 
\begin{split}
 \Delta_{x'} h \in W^{1/2-1/(2r)}_r(0,T;L_r(\Sigma)) \cap L_r(0;T;W^{1-1/r}_r(\Sigma)).
 \end{split}
\end{equation}
Furthermore, there is some $C = C(T) > 0$, such that
\begin{equation} \label{6087698}
\begin{alignedat}{1}
| \Delta_{x'} h &|_{ W^{1/2-1/(2r)}_r(0,T;L_r(\Sigma)) \cap L_r(0;T;W^{1-1/r}_r(\Sigma)) }  \\ &\leq C(T) |h|_{F^{3/2-1/2q}_{pq}(0,T; L_q (\Sigma))  \cap F^{1-1/2q}_{pq}(0,T; H^2_q(\Sigma))  \cap L_p(0,T ; W^{4-1/q}_q(\Sigma))}.
\end{alignedat}
\end{equation}
Furthermore,
\begin{equation} \label{4354ffff35}
H^1_r(0,T;L_r(\Omega)) \cap L_r(0,T;H^2_r(\Omega \backslash \Sigma)) \into BUC(0,T;C^1(\Omega \backslash \Sigma)).
\end{equation}
Also for $v \in H^1_r(0,T;L_r(\Omega)) \cap L_r(0,T;H^2_r(\Omega \backslash \Sigma)) \cap \{ v : v|_{t=0} = 0 \}$, the Navier-Stokes part is contractive in the Stefan-problem, that is, there is some $\epsilon > 0$ such that
\begin{equation}
| v|_\Sigma |_{ F^{1/2-1/2q}_{pq}(0,T; L_q(\Sigma)) \cap L_p(0;T; W^{1-1/q}_q(\Sigma)) } \leq CT^\epsilon |v|_{_0 H^1_r(0,T;L_r(\Omega)) \cap L_r(0,T;H^2_r(\Omega \backslash \Sigma))}
\end{equation}
for some constant $C>0$ independent on $T$.

By restricting to height functions $h$ with initial trace zero, $h(0) = 0$, the embedding constant in \eqref{6087698} can be chosen to be independent of $T$ and only depending on $T_0$. In particular, the embedding does not degenerate and the embedding constant stays bounded as $T \pfeil 0$.
 \end{lemma}
 \begin{proof}
 The proof follows the lines of the proof of Theorem 4.1 in \cite{navmulsekpap}. We shall give the modifications to it. It is shown there that $$L_p(0,T;W^{2-1/q}_q(\Sigma))\into L_r(0,T;W^{1-1/r}_r(\Sigma)),$$ provided $r \leq p$ and $r < 3q/(3-q)$. Hence $\Delta h \in L_r(0,T; W^{1-1/r}_r(\Sigma))$.
 Let us focus on the time regularity of $\Delta h$. By Proposition 5.38 in \cite{kaipdiss},
 \begin{align*}
 F^{1-1/2q}_{pq}&(\R_+; H^2_q (\R^d))  \cap L_p(\R_+ ; B^{4-1/q}_{qq}(\R^d)) \into \\
 &\into 
 H^{\theta(1-1/2q)}_{p}(\R_+; B^{4-1/q - \theta(2-1/q)}_{qq} (\R^d)),
 \end{align*}
 for any $\theta \in (0,1)$. Hence $\Delta h \in H^{\theta(1-1/2q)}_{p}(\R_+; B^{2-1/q - \theta(2-1/q)}_{qq} (\R^d))$, $\theta \in (0,1)$. Now, $B^{2-1/q - \theta(2-1/q)}_{qq} (\R^d) \into L_r(\R^d)$, provided
 \begin{equation} \label{34590349053495jj345}
 \theta < \frac{2-3/q+2/r}{2-1/q}.
 \end{equation}
 Then $\theta(1-1/2q) < 1 - 3/2q + 1/r$. We now see that $1 - 3/2q + 1/r > 1/2 - 1/2r$ is equivalent to $ r < 3q/(3-q)$. Hence this inequality can easily be achieved for $q \in (19,10)$ and $r > 5$.
 We obtain that $ \Delta h \in W^{1/2-1/2r}( \R_+ ; L_r(\R^d))$ by choosing $\theta \in (0,1)$ satisfying
 $1/2-1/2r<\theta(1-1/2q) < 1 - 3/2q + 1/r$. Then by standard extension and restriction arguments the first part of the lemma follows.
 
 Now we take a function $v \in H^1_r(0,T;L_r(\Omega^+)) \cap L_r(0,T; H^2_r(\Omega^+))$ for $r > 5$. We obtain that $v|_\Sigma \in BUC([0,T] ; W^{2-3/r}_r(\Sigma)) \into BUC([0,T] ; C^1(\Sigma))$, since $ r> 5$. Hence already
 \begin{equation}
 | v|_\Sigma|_{L_p(0,T; W^{1-1/q}_q(\Sigma))} \leq T^{1/p} |v|_{ H^1_r(0,T;L_r(\Omega)) \cap L_r(0,T; H^2_r(\Omega \backslash \Sigma)) }.
 \end{equation}
 It remains to show that $v|_\Sigma$ is contractive in $F^{1/2 - 1/2q}_{pq}(0,T; W^{1-1/q}_q(\Sigma))$, provided that $v|_{t=0} = 0$. We need this restriction here since $v|_\Sigma$ has a time-trace at $t = 0$ and we want the estimates not to degenerate as $T \pfeil 0$. To use the results of \cite{kaipdiss} we need to extend $v$ in time to the half line $\R_+$ by reflection. 
 Now, for such a $v$ we have that
 \begin{equation}
 v \in H^1_s(0,T;L_r(\Omega)) \cap L_s(0,T; H^2_r(\Omega \backslash \Sigma))
 \end{equation}
 for some $5 < s < r$ and
 \begin{equation}
 |v|_{ H^1_s(0,T;L_r(\Omega)) \cap L_s(0,T; H^2_r(\Omega \backslash \Sigma)) } \leq T^{\frac{r-s}{rs} }|v|_{ H^1_r(0,T;L_r(\Omega)) \cap L_r(0,T; H^2_r(\Omega \backslash \Sigma))}.
 \end{equation}
 Furthermore,
 \begin{equation}
 v \in H^{1/2}_s(0,T; H^1_s(\Omega \backslash \Sigma)) \into H^{1/2}_s(0,T; H^1_q(\Omega \backslash \Sigma)),
 \end{equation}
 since $q < s <r$. Taking traces, $v \in H^{1/2}_s(0,T; W^{1-1/q}_q( \Sigma))$. Extending $v$ to the half line, then using \cite{kaipdiss} and the fact that $1/2 - 1/s > 1/2 - 1/2q - 1/p$ and $r > 5$ gives the embedding $H^{1/2}_s(\R_+; W^{1-1/q}_q( \Sigma)) \into F^{1/2 - 1/2q}_{pq} (\R_+ ; W^{1-1/q}_q( \Sigma))$. 
 \end{proof}
 
 As for the Stefan problem \eqref{34905384953745034875}, we assume that the free boundary $\Gamma(t)$ is a graph of a function $h$ over $\Sigma$. We transform the coupled Navier-Stokes/Stefan problem \eqref{34ghfghfgh75} to the fixed reference configuration $\Omega \backslash \Sigma$ by means of the Hanzawa transform \eqref{4jjjjjgjfjgjjfgjfjgfj}. Let $\rho = \rho^+ \chi_+ + \rho^- \chi_-$, $\mu = \mu^+ \chi_+ + \mu^- \chi_-$, where $\chi_\pm$ is the indicator function of $\Omega^\pm := \Omega \cap \{ x_3 \gtrless 0 \}$.   The problem for the transformed quantities $(\bar v, \bar p ,\bar u)$ then reads as
  \begin{equation} \label{9304hhhhhhnnnnnnbbbbb7gg5} \begin{cases}
\begin{alignedat}{2} 
\rho \partial_t  \bar v  - \mu \Delta \bar v  + \nabla \bar p &= F_v(\bar v, \bar p, h) , &\text{in } \Omega \backslash \Sigma, \\
\operatorname{div} \bar v &= F_d(\bar v, h), &\text{in } \Omega \backslash \Sigma, \\
- \ljump \mu (Dv + Dv^\top) -\bar p I \rjump \nu_{\Sigma }   &=  \sigma \Delta_{x'} h \nu_\Sigma + F_S(\bar v, \bar p, h), & \text{on } \Sigma,  \\
\ljump \bar v \rjump &= 0, & \text{on } \Sigma, \\
\partial_t h &= \bar v \cdot \nu_\Sigma - \ljump \partial_3 \bar u \rjump + F_\Sigma(\bar u, \bar v, h), & \text{on } \Sigma, \\
(-\nabla_{x'} h, 0)^\top \cdot \nu_{S_1} &= 0, & \text {on } \partial \Sigma, \\
(\partial_t - \Delta) \bar u &= F_c(\bar u, h), &\text {in } \Omega \backslash \Sigma, \\
\bar u|_{\Sigma} - \sigma \Delta_{x'} h &= F_\kappa (h), & \text {on } \Sigma, \\
\nu_{\partial\Omega} \cdot \nabla \bar u|_{\partial\Omega} &= F_N(\bar u, h), & \text{on } \partial\Omega \backslash\Sigma, \\
P_{S_1} \left( \mu (D\bar v + D\bar v^\top) \nu_{S_1} \right) &= F_P^\pm (\bar v, h), & \text{on } S_1 \backslash \partial\Sigma, \\
\bar v \cdot \nu_{S_1} &= 0, & \text {on } S_1 \backslash \partial \Sigma, \\
\bar v &= 0, & \text{on } S_2, \\
v(0) &= v_0, & \text{on } \Omega \backslash \Sigma, \\
u(0) &= u_0, & \text{on } \Omega \backslash \Sigma, \\
h (0) &= h_0, & \text{on } \Sigma, 
\end{alignedat} \end{cases}
\end{equation}
where $\nu_\Sigma = e_3$, and the nonlinearities on the right hand side are
\begin{equation}\label{7hhghhhjhgjjghjj643}
\begin{alignedat}{1} 
F_v(\bar v, \bar p, h) &:= \mu (\Delta_h - \Delta)\bar v + (\nabla - \nabla_h)\bar p +  D \bar v \cdot \partial_t \Theta_h^{-1} - (\bar v \cdot \nabla_h )\bar v , \\
F_d(h,\bar v) &:= (\div - \div_h) \bar v, \\
F_S (h,\bar v, \bar p) &:=   \ljump \mu^\pm \left( (D\Theta_h - I)D \bar v + D \bar v^\top (D\Theta_h - I)^\top ) \right) \rjump \nu_{\Gamma_h} +    \\
&\quad+ \ljump \left( \mu^\pm (D \bar v+D \bar v^\top) - \bar p I \right) (e_3 - \nu_{\Gamma_h} ) \rjump + \sigma( K(h) \nu_{\Gamma_h} - \Delta h e_3 ), \\
F_P^\pm(h,\bar v) &:= P_{S_1} \left( \mu^\pm \left( (D\Theta_h - I)D\bar v + D \bar v ^\top (D\Theta_h - I)^\top ) \right) \nu_{S_1} \right), \\
F_u (\bar u,h) &:=  (\Delta_h - \Delta)\bar u + D \bar u \cdot \partial_t \Theta_h^{-1}, \\
F_\kappa(h) &:= \sigma \left[K(h) - \Delta h \right], \\
F_N (\bar u,h) &:= (\nu_{\partial\Omega} | (\nabla-\nabla_h) \bar u) , \\
F_\Sigma (\bar v, \bar u, h) &:= \ljump \partial_3 \bar u - n_\Gamma \cdot \nabla_h \bar u \rjump + \partial_t h (e_3| e_3 - n_\Gamma ) + (\bar v | e_3 - \nu_\Gamma ).
\end{alignedat}
\end{equation}
Here $K(h)$ is given by
\begin{equation}
K(h) = \div \left( \frac{\nabla h}{\sqrt { 1+ |\nabla h|^2 }} \right).
\end{equation}
 Let us now show maximal regularity for the linearization.
 For convenience we will drop the bars and consider the linear problem
   \begin{equation} \label{9304hlienarporbjsdfsdfsdfkdfjgdfgg7gg5} \begin{cases}
\begin{alignedat}{2} 
\rho \partial_t   v  - \mu \Delta v  + \nabla  p &= f_v , &\text{in } \Omega \backslash \Sigma, \\
\operatorname{div}  v &= f_d, &\text{in } \Omega \backslash \Sigma, \\
- \ljump \mu (Dv + Dv^\top) - p I \rjump \nu_{\Sigma } - \sigma \Delta h \nu_\Sigma   &=  f_S, & \text{on } \Sigma,  \\
\ljump  v \rjump &= f_J, & \text{on } \Sigma, \\
\partial_t h - v^+|_\Sigma \cdot e_3 + \ljump \partial_3  u \rjump &=  f_\Sigma, & \text{on } \Sigma, \\
(-\nabla_{x'} h, 0)^\top \cdot \nu_{S_1} &= f_a, & \text {on } \partial \Sigma, \\
(\partial_t - \Delta)  u &= f_w, &\text {in } \Omega \backslash \Sigma, \\
\ljump u \rjump =0, \quad u|_{\Sigma} - \sigma \Delta h &= f_\kappa, & \text {on } \Sigma, \\
\nu_{\partial\Omega} \cdot \nabla  u|_{\partial\Omega} &= f_n, & \text{on } \partial\Omega \backslash\Sigma, \\
P_{S_1} \left( \mu (D v + D v^\top) \nu_{S_1} \right) &= P_{S_1} f_P, \quad\quad & \text{on } S_1 \backslash \partial\Sigma, \\
 v \cdot \nu_{S_1} &= f_N, & \text {on } S_1 \backslash \partial \Sigma, \\
v &= f_D, & \text{on } S_2, \\
v(0) &= v_0, & \text{on } \Omega \backslash \Sigma, \\
u(0) &= u_0, & \text{on } \Omega \backslash \Sigma, \\
h (0) &= h_0, & \text{on } \Sigma, 
\end{alignedat} \end{cases}
\end{equation}
where $v^+ := v|_{\Omega^+}$ and the right hand side are all given data. We are interested in strong solutions
\begin{equation} \label{3445jjjh3j4k5h34jk5}
\begin{gathered}
u \in H^1_p(0,T; L_q(\Omega)) \cap L_p(0,T; H^2_q(\Omega \backslash \Sigma)) , \\
h \in F^{3/2-1/2q}_{pq}(0,T; L_q (\Sigma))  \cap F^{1-1/2q}_{pq}(0,T; H^2_q(\Sigma))  \cap L_p(0,T ; W^{4-1/q}_q(\Sigma)), \\
v \in H^1_r(0;T; L_r(\Omega)) \cap L_r(0,T; H^2_r(\Omega \backslash \Sigma)), \\
p \in L_r(0,T; \dot H^1_r (\Omega \backslash \Sigma)), \\
\ljump p \rjump \in W^{1/2 - 1/2r}_r(0,T;L_r(\Sigma)) \cap L_r(0,T;W^{1-1/r}_r(\Sigma)).
\end{gathered}
\end{equation}
Let us now state necessary conditions for the data. These consist of the conditions for the Stefan problem and the Stokes problem in a capillary, cf. \cite{wilkehabil}. We have the regularity conditions
\begin{equation} \label{345kohklhjklhlhj345}
\begin{gathered}
f_v \in L_r(0,T; L_r (\Omega)), \\
f_d \in L_r(0,T; H^1_r(\Omega \backslash \Sigma)), \\
f_S \in  W^{1/2-1/(2r)}_r(0,T;L_r(\Sigma)) \cap L_r(0;T;W^{1-1/r}_r(\Sigma)), \\
f_J \in  W^{1-1/(2r)}_r(0,T;L_r(\Sigma)) \cap L_r( 0,T; W^{2-1/r}_r(\Sigma)),  \\
f_\Sigma \in F^{1/2-1/2q}_{pq}(0,T; L_q(\Sigma)) \cap L_p(0,T;W^{1-1/q}_q(\Sigma)), \\
f_a \in {F}^{5/4 - 1/q + 1/4(3q-1)   }_{pq}(0,T; L_q (\partial \Sigma))  \cap {F}^{1-1/2q}_{pq}(0,T; W^{1-1/q}_q(\partial \Sigma)) \\ \quad\quad\quad \cap \;  L_p(0,T ; W^{3-2/q}_q(\partial \Sigma)),
\end{gathered}
\end{equation}
as well as
\begin{equation}
\begin{gathered}
f_w \in L_p(0,T;L_q(\Omega)), \\
f_\kappa \in F^{1-1/2q}_{pq}(0,T; L_q(\Sigma)) \cap  L_p(0,T;W^{2-1/q}_q(\Sigma)), \\
f_n \in   F^{1/2-1/2q}_{pq}(0,T; L_q(\partial \Omega)) \cap   L_p(0,T; W^{1-1/q}_q(\partial\Omega)), \\
P_{S_1}f_P \in  W^{1/2-1/(2r)}_r(0,T;L_r(S_1)) \cap L_r(0;T;W^{1-1/r}_r(S_1)),\\
f_N \in W^{1-1/(2r)}_r(0,T;L_r(S_1)) \cap L_r( 0,T; W^{2-1/r}_r(S_1 \backslash \partial\Sigma)),\\
f_D \in W^{1-1/(2r)}_r(0,T;L_r(S_2)) \cap L_r( 0,T; W^{2-1/r}_r(S_2)), \\
v_0 \in W^{2-2/r}_{r} (\Omega \backslash \Sigma) , \quad h_0 \in B^{4-1/q-2/p}_{qp}(\Sigma) , \quad u_0 \in W^{2-2/q}_q(\Omega \backslash \Sigma), \\
(f_d,f_J,f_N,f_D) \in H^1_r(\R_+; \hat H^{-1}_r(\Omega)), 
\end{gathered}
\end{equation}
the compatibility conditions at time $t = 0$,
\begin{equation} \label{345kohklhjklhlhj345B}
\begin{gathered}
\div v_0 = f_d|_{t=0} , \quad \text{in } \Omega \backslash \Sigma, \\
- \ljump \mu^\pm \partial_3 (v_0)_{1,2} \rjump - \ljump \mu^\pm \nabla_{x'} (v_0)_3 \rjump = (f_S)_{1,2}|_{t=0}, \quad\text{on } \Sigma, \\
\ljump v_0 \rjump = f_J|_{t=0}, \quad\text{on } \Sigma,\\
P_{S_1} ( \mu^\pm (Dv_0 + Dv_0^\top)\nu_{S_1}) = P_{S_1} f_P|_{t=0}, \quad\text{on } S_1, \\
v_0 \cdot \nu_{S_1} = f_N|_{t=0}, \quad\text{on } S_1, \\
v_0|_{S_2} = f_D|_{t=0}, \quad\text{on } S_2, \\
(-\nabla_{x'} h_0, 0)^\top \cdot \nu_{S_1} = f_a|_{t=0}, \quad\text{on } \partial \Sigma,  \\
n_{\partial \Omega} \cdot \nabla u_0 = f_n|_{t=0}, \quad \text{on } \partial\Omega, \\
\ljump u_0 \rjump = 0, \quad \text{on } \Sigma, \\
u_0|_\Sigma - \Delta h_0 = f_\kappa (0), \quad \text{on } \Sigma, \\
\ljump \partial_3 u_0 \rjump - f_h (0) \in B^{2-2/q-4/p}_{qp}(\Sigma),
\end{gathered}
\end{equation}
and conditions on commonly shared boundaries of the capillary,
\begin{equation} \label{345kohklhjklhlhj345C}
\begin{gathered}
\ljump f_N \rjump  = f_J \cdot \nu_{S_1}, \quad  \text{on } \partial \Sigma, \\
\ljump ( f_P \cdot e_3)/{\mu^\pm} - \partial_3 f_N \rjump = \partial_{\nu_{S_1}} (f_J \cdot e_3), \quad \text{on } \partial\Sigma, \\
P_{\partial\Sigma} [ ( D_{x'} \Pi f_J + (D_{x'}\Pi f_J)^\top )\nu_{\partial\Sigma}] = \ljump P_{\partial\Sigma}\Pi f_P /{\mu^\pm} \rjump, \quad \text{on } \partial\Sigma, \\
(f_S)_{1,2} \cdot (\nu_{S_1})_{1,2} = - \ljump f_P \cdot e_3 \rjump,  \quad \text{on } \partial\Sigma, \\
f_D \cdot \nu_{S_1} = f_N,  \quad \text{on } \partial{S_2}, \\
P_{\partial\Sigma}[  \mu^\pm ( D_{x'} \Pi f_D + (D_{x'} \Pi f_D)^\top)\nu_{\partial\Sigma} ] = P_{\partial\Sigma} \Pi f_P,\quad \text{on } \partial S_2, \\
\mu^\pm \partial_{\nu_{S_1}}(f_D \cdot e_3) + \mu^\pm \partial_3 f_N = f_P \cdot e_3, \quad \text{on } \partial S_2. 
\end{gathered}
\end{equation}
The next result states maximal regularity for the linear problem \eqref{9304hlienarporbjsdfsdfsdfkdfjgdfgg7gg5}.
\begin{theorem} \label{3489053485jjj345}
Let $\mu^\pm, \rho^\pm, \sigma >0$ constant, $(p,q,r)$ as in Lemma \ref{kl3j4h5jh345345} and $0 < T < \infty$. Then for every
$(f_v, f_d, f_S, f_J, f_\Sigma, f_a, f_w, f_\kappa, f_n, f_P, f_N, f_D, v_0, u_0, h_0)$ satisfying \eqref{345kohklhjklhlhj345}, \eqref{345kohklhjklhlhj345B} and \eqref{345kohklhjklhlhj345C}, there is unique $(v,u,p, \ljump p \rjump, h)$ in the regularity classes of \eqref{3445jjjh3j4k5h34jk5} solving \eqref{9304hlienarporbjsdfsdfsdfkdfjgdfgg7gg5} on $(0,T)$. Furthermore, the solution map
$$[(f_v, f_d, f_S, f_J, f_\Sigma, f_a, f_w, f_\kappa, f_n, f_P, f_N, f_D, v_0, u_0, h_0) \mapsto (v,u,p, \ljump p \rjump, h)]$$ is continuous. Furthermore, the norm of the solution map with trivial initial values
$[(f_v, f_d, f_S, f_J, f_\Sigma, f_a, f_w, f_\kappa, f_n, f_P, f_N, f_D, 0, 0, 0) \mapsto (v,u,p, \ljump p \rjump, h)]$ is independent of $T$ thanks to Lemma \ref{kl3j4h5jh345345}.
\end{theorem}
\begin{proof}
For the proof we refer to the proof of Theorem 4.4 in \cite{navmulsekpap}. We first reduce to trivial initial values $(v_0, u_0, h_0) = (0,0,0)$ and resolve the data. Then by Lemma \ref{kl3j4h5jh345345} we can use maximal regularity for the Stefan problem and apply a perturbation argument as in \cite{navmulsekpap}. It is noteworthy that $v^+|_\Sigma \in C^0([0,T] ; C^1 (\Sigma))$ by the embeddings since $r > 5$. Since $C^1(\Sigma) \into B^{2-2/q-4/p}_{qp}(\Sigma)$ because of $q < 2$, $v^+(0) \in B^{2-2/q-4/p}_{qp}(\Sigma)$ and the compatibility $\eqref{345kohklhjklhlhj345B}_{11}$ is enough to reduce to trivial initial values of $h$ and $\partial_t h$.
\end{proof}
\subsection*{Nonlinear well-posedness.}
We consider the full nonlinear problem \eqref{9304hhhhhhnnnnnnbbbbb7gg5}.
\begin{theorem}

Let $p,q ,r$ as in Lemma \ref{kl3j4h5jh345345}. Then there is some $\delta > 0$, such that if
\begin{equation}
|v_0 |_{W^{2-2/r}_r(\Omega \backslash \Sigma)} + |u_0 |_{W^{2-2/q}_q(\Omega \backslash \Sigma)} + |h_0|_{B^{4-1/q - 2/p}_{qp}(\Sigma)} \leq \delta
\end{equation}
and $(v_0,u_0,h_0)$ satisfy the compatibility conditions
\begin{equation} \label{345kofdgdfgdfg45B}
\begin{gathered}
\div v_0 = F_d (v_0,h_0) , \quad \text{in } \Omega \backslash \Sigma, \\
- \ljump \mu^\pm \partial_3 (v_0)_{1,2} \rjump - \ljump \mu^\pm \nabla_{x'} (v_0)_3 \rjump = (F_S)_{1,2}(v_0,h_0), \quad\text{on } \Sigma, \\
\ljump v_0 \rjump = 0, \quad\text{on } \Sigma,\\
P_{S_1} ( \mu^\pm (Dv_0 + Dv_0^\top)\nu_{S_1}) = 0, \quad\text{on } S_1, \\
v_0 \cdot \nu_{S_1} = 0, \quad\text{on } S_1, \\
v_0|_{S_2} = 0, \quad\text{on } S_2, \\
(-\nabla_{x'} h_0, 0)^\top \cdot \nu_{S_1} = 0, \quad\text{on } \partial \Sigma,  \\
n_{\partial \Omega} \cdot \nabla u_0 = F_N(u_0,h_0), \quad \text{on } \partial\Omega, \\
\ljump u_0 \rjump = 0, \quad \text{on } \Sigma, \\
u_0|_\Sigma - \Delta h_0 = F_\kappa (h_0), \quad \text{on } \Sigma, \\
\ljump \partial_3 u_0 \rjump - f_h (0) \in B^{2-2/q-4/p}_{qp}(\Sigma),
\end{gathered}
\end{equation}
the transformed two-phase Navier-Stokes/Stefan problem \eqref{9304hhhhhhnnnnnnbbbbb7gg5} possesses a unique strong solution $(v,p,u,h)$ in the sense of \eqref{3445jjjh3j4k5h34jk5} on $(0, \tau)$ for some $\tau > 0$.
\end{theorem}
\subsection*{Qualitative behaviour.}
Let us now again investigate the long-time behaviour.


For simplicity, let $\rho = 1$. By testing $\eqref{34ghfghfgh75}_1$ with the solution $v$ and invoking the other equations we obtain
\begin{equation}
\frac{d}{dt} \left[ \int_\Omega \frac{|v|^2}{2} \int_{\Gamma(t)} \sigma + \int_\Omega \frac{|u|^2}{2} \right] = - 2 \int_\Omega \mu | \mathbb D v |^2 - \int_\Omega |\nabla u|^2.
\end{equation}
Hence any stationary solution $(v,p,u,\Gamma)$ satisfies that $v = 0$, $p$ is constant with possibly different values in the phases, and $u$ and $H_\Gamma$ are constant. Note that the pressure $p$ can always be reconstructed by solving a weak transmission problem. The set of equilibrium solutions is
\begin{equation}
\mathcal E = \{ (v,u, \Gamma) : v= 0, \; H_\Gamma = const., \; u = \sigma H_\Gamma \}.
\end{equation}
If we now additionally assume that $\Gamma$ is the graph of a function $h$ over $\Sigma$, we again obtain $H_\Gamma = 0$. In particular, $h$ has to be constant then and $u = 0$.

We will now study the problem for the height function \eqref{345345345435345345} in an $L_p$-setting. The equilibria in the graph case are
\begin{equation}
\mathcal E_\Sigma = \{ (v,u, h) : v = 0, \; u = 0, \; h = const. \}.
\end{equation}
The linearization of the (transformed) two-Phase Navier-Stokes/Stefan problem with Gibbs-Thomson correction around the trivial equilibrium $(v_*,u_*,h_*) = (0,0,0)$ motivates us to study the linear problem
 \begin{equation} \label{93ccc0ggfdgbtztz7gg5} \begin{cases}
\begin{alignedat}{2} 
 \partial_t   v  - \mu \Delta v  + \nabla  p &= f_v , &\text{in } \Omega \backslash \Sigma, \\
\operatorname{div}  v &= 0, &\text{in } \Omega \backslash \Sigma, \\
- \ljump \mu (Dv + Dv^\top) - p I \rjump \nu_{\Sigma } - \sigma \Delta h \nu_\Sigma   &=  0, & \text{on } \Sigma,  \\
\ljump  v \rjump &= 0, & \text{on } \Sigma, \\
\partial_t h - v|_\Sigma \cdot e_3 + \ljump \partial_3  u \rjump &=  f_h, & \text{on } \Sigma, \\
(-\nabla_{x'} h, 0)^\top \cdot \nu_{S_1} &= 0, & \text {on } \partial \Sigma, \\
(\partial_t - \Delta)  u &= f_u, &\text {in } \Omega \backslash \Sigma, \\
\ljump u \rjump =0, \quad u|_{\Sigma} - \sigma \Delta h &= 0, & \text {on } \Sigma, \\
\nu_{\partial\Omega} \cdot \nabla  u|_{\partial\Omega} &= 0, & \text{on } \partial\Omega \backslash\Sigma, \\
P_{S_1} \left( \mu (D v + D v^\top) \nu_{S_1} \right) &= 0, \qquad\qquad\quad & \text{on } S_1 \backslash \partial\Sigma, \\
 v \cdot \nu_{S_1} &= 0, & \text {on } S_1 \backslash \partial \Sigma, \\
v &= 0, & \text{on } S_2, \\
v(0) &= v_0, & \text{on } \Omega \backslash \Sigma, \\
u(0) &= u_0, & \text{on } \Omega \backslash \Sigma, \\
h (0) &= h_0, & \text{on } \Sigma, 
\end{alignedat} \end{cases}
\end{equation}
We now rewrite \eqref{93ccc0ggfdgbtztz7gg5} as an abstract evolution equation. Define 
\begin{equation}
X_0 := L_r(\Omega) \times L_q (\Omega) \times W^{2-2/q}_q(\Sigma), \quad X_1 := H^2_r(\Omega \backslash \Sigma) \times H^2_q(\Omega \backslash \Sigma) \times W^{4-1/q}_q(\Sigma),
\end{equation}
and the linear operator $B$ in $X_0$ by means of $B : D(B) \subset X_1 \pfeil X_0$,
\begin{equation} \label{345kkkkkj34534543}
B(v,u,h) := (-\Delta v + \nabla p, -\Delta u, - v|_\Sigma \cdot \nu_\Sigma + \ljump \partial_3 u \rjump ),
\end{equation}
with domain
\begin{equation}
\begin{alignedat}{1}
D(B) := \{ (v,&u,h) \in X_1 :  \ljump u \rjump =0 \text{ on } \Sigma, \; u|_\Sigma = \sigma \Delta h \text{ on } \Sigma, \; \ljump \partial_3 u \rjump \in W^{2-2/q}_q(\Sigma),  \\
&\ljump v \rjump =0  \text{ on } \Sigma, \; P_{S_1} \left( \mu^\pm (Dv + Dv^\top) \nu_{S_1} \right) = 0 \text{ on } S_1 \backslash \partial \Sigma,\\ & v \cdot \nu_{S_1} = 0 \text{ on } S_1, v|_{S_2} = 0, \; P_\Sigma ( \ljump \mu^\pm (Dv + Dv^\top) \rjump \nu_\Sigma ) = 0 \text{ on } \Sigma
\}
\end{alignedat}
\end{equation}
In equation \eqref{345kkkkkj34534543}, the pressure $p \in \dot H^1_r (\Omega \backslash \Sigma)$ solves the weak transmission problem
\begin{align*}
(\nabla p | \nabla \phi)_{L_2(\Omega)} &= (\mu \Delta v  | \nabla \phi )_{L_2(\Omega)}, \quad & \text{for all } \phi \in W^1_{r'}(\Omega), \\
\ljump p \rjump &= \sigma \Delta h + ( \ljump \mu^\pm (Dv+Dv^\top)  \rjump \nu_\Sigma | \nu_\Sigma )_{L_2(\Sigma )}, & \text{on } \Sigma,
\end{align*}
cf. Lemma A.7 in \cite{wilkehabil}.

For $f_v \in L_r(0,T;L_r(\Omega))$, $f_h \in L_p(0,T; W^{2-2/q}_q(\Sigma))$, and $f_u \in L_p(0,T; L_q (\Omega))$, we may rewrite \eqref{93ccc0ggfdgbtztz7gg5} as an abstract evolution equation in $X_0$,
\begin{equation}
\dot z (t) + Bz(t) = f(t), \; t > 0, \quad z(0) = z_0,
\end{equation}
where $f := (f_v,f_h, f_u)$ and $z_0 := (v_0, u_0, h_0)$. The operator $B$ now has the following properties.
\begin{lemma}
Let $n = 2,3$, $p,q,r$ as in Lemma \ref{kl3j4h5jh345345}, $B$, $X_0$ as above and $z_* = (0,0,0)$ be the trivial equilibrium.
\begin{enumerate}
\item The linear operator $-B$ generates an analytic $C_0$-semigroup $e^{-Bt}$ in $X_0$ and
the spectrum $\sigma(-B)$ consists of at most countably many eigenvalues with finite algebraic multiplicity.
 \item $\sigma(-B) \cap i\R \subset \{ 0 \}$ and $\sigma(-B) \backslash \{ 0 \} \subset \mathbb C_- := \{ z \in \mathbb C : \operatorname{Re} z < 0 \}$. 
\item $\lambda = 0$ is semi-simple with multiplicity one, $X_0 = N(B) \oplus R(B)$. 
\item $N(B) = T_{z_*} \mathcal E_\Sigma = \mathcal E_\Sigma$ and $N(B)$ is spanned by $(0,0,1)$.
\item The restriction of the semigroup $e^{-Bt}$ to ${R(B)}$ is exponentially stable.
\end{enumerate}
\end{lemma}
\begin{proof}
By compact embeddings, $\sigma(B)$ only consists of countably many eigenvalues with finite multiplicity. Let $\lambda \in \sigma(-B)$ with eigenfunctions $(v,u,h)$. The corresponding eigenvalue problem reads
\begin{equation} \label{93ccc0ffdfvdfvvbvbg5} \begin{cases}
\begin{alignedat}{2} 
 \lambda   v  - \mu \Delta v  + \nabla  p &= 0 , &\text{in } \Omega \backslash \Sigma, \\
\operatorname{div}  v &= 0, &\text{in } \Omega \backslash \Sigma, \\
- \ljump \mu (Dv + Dv^\top) - p I \rjump \nu_{\Sigma } - \sigma \Delta h \nu_\Sigma   &=  0, & \text{on } \Sigma,  \\
\ljump  v \rjump &= 0, & \text{on } \Sigma, \\
\lambda h - v \cdot \nu_\Sigma + \ljump \partial_3  u \rjump &=  0, & \text{on } \Sigma, \\
(-\nabla_{x'} h, 0)^\top \cdot \nu_{S_1} &= 0, & \text {on } \partial \Sigma, \\
(\lambda - \Delta)  u &= 0, &\text {in } \Omega \backslash \Sigma, \\
\ljump u \rjump =0, \quad u|_{\Sigma} - \sigma \Delta h &= 0, & \text {on } \Sigma, \\
\nu_{\partial\Omega} \cdot \nabla  u|_{\partial\Omega} &= 0, & \text{on } \partial\Omega \backslash\Sigma, \\
P_{S_1} \left( \mu (D v + D v^\top) \nu_{S_1} \right) &= 0, \qquad\qquad\quad & \text{on } S_1 \backslash \partial\Sigma, \\
 v \cdot \nu_{S_1} &= 0, & \text {on } S_1 \backslash \partial \Sigma, \\
v &= 0, & \text{on } S_2.
\end{alignedat} \end{cases}
\end{equation}
Testing $\eqref{93ccc0ffdfvdfvvbvbg5}_1$ with $v$ and invoking the other equations, we obtain
\begin{equation} \label{345kkkjhjk345}
\lambda |v|_{L_2(\Omega)}^2 + | \mu^{1/2} \mathbb D v |_{L_2(\Omega)}^2 + \sigma \bar \lambda | \nabla h |^2_{L_2(\Sigma)} + | Du|_{L_2(\Omega)}^2 + \lambda |u|_{L_2(\Omega)}^2 = 0.
\end{equation}
If $\lambda = 0$, we obtain that $u$ is constant and $v = 0$ by Korn's inequality. By $\eqref{93ccc0ffdfvdfvvbvbg5}_8$ also $\Delta h$ is constant. Integrating $\Delta h$ over $\Sigma$ and invoking the boundary condition $\eqref{93ccc0ffdfvdfvvbvbg5}_6$ gives $\Delta h = 0$, which in turn implies that $h$ is constant. Equation $\eqref{93ccc0ffdfvdfvvbvbg5}_8$ then renders $u = 0$. This implies that $N(B)$ is spanned by $(0,0,1)$.

Taking real parts in \eqref{345kkkjhjk345} gives that any eigenvalue $\lambda$ satisfies $\operatorname{Re} \lambda \leq 0$. Furthermore, if an eigenvalue $\lambda$ has real part zero, again $\eqref{345kkkjhjk345}$ yields that $u$ is constant and $v = 0$ again by Korn's inequality. Equation $\eqref{93ccc0ffdfvdfvvbvbg5}_5$ gives $\lambda h = 0$, and since $h$ may not be trivial, $\lambda = 0$.

Let us show that $N(B) = N(B^2)$. Pick some $(v,u,h) \in N(B^2)$ and define $(v_1,u_1,h_1) := B(v,u,h)$. Then $(v_1,u_1,h_1) \in N(B)$, hence $v_1 = u_1 = 0$ and $h_1$ is constant. The problem for $(v,u,h)$ reads
\begin{equation} \label{93cccggdfvggcxxcvvvg5} \begin{cases}
\begin{alignedat}{2} 
  - \mu \Delta v  + \nabla  p &= 0 , &\text{in } \Omega \backslash \Sigma, \\
\operatorname{div}  v &= 0, &\text{in } \Omega \backslash \Sigma, \\
- \ljump \mu (Dv + Dv^\top) - p I \rjump \nu_{\Sigma } - \sigma \Delta h \nu_\Sigma   &=  0, & \text{on } \Sigma,  \\
\ljump  v \rjump &= 0, & \text{on } \Sigma, \\
 - v \cdot \nu_\Sigma + \ljump \partial_3  u \rjump &=  h_1, & \text{on } \Sigma, \\
(-\nabla_{x'} h, 0)^\top \cdot \nu_{S_1} &= 0, & \text {on } \partial \Sigma, \\
  \Delta  u &= 0, &\text {in } \Omega \backslash \Sigma, \\
\ljump u \rjump =0, \quad u|_{\Sigma} - \sigma \Delta h &= 0, & \text {on } \Sigma, \\
\nu_{\partial\Omega} \cdot \nabla  u|_{\partial\Omega} &= 0, & \text{on } \partial\Omega \backslash\Sigma, \\
P_{S_1} \left( \mu (D v + D v^\top) \nu_{S_1} \right) &= 0, \qquad\qquad\quad & \text{on } S_1 \backslash \partial\Sigma, \\
 v \cdot \nu_{S_1} &= 0, & \text {on } S_1 \backslash \partial \Sigma, \\
v &= 0, & \text{on } S_2.
\end{alignedat} \end{cases}
\end{equation}
Integrating $\eqref{93cccggdfvggcxxcvvvg5}_5$ over $\Sigma$, an integration by parts yields
\begin{equation}
0 = -\int_{\Omega^+} \div v dx + \int_\Omega \Delta u dx = -\int_\Sigma v \cdot \nu_\Sigma dx' + \int_\Sigma \ljump \partial_3 u \rjump dx' = \int_\Sigma h_1 dx' = h_1 | \Sigma|.
\end{equation}
Hence $h_1 = 0$ and $(v,u,h) \in N(B)$. The rest follows as in the proof of Lemma \ref{3445kk34k5jhl345}.
\end{proof}
\begin{remark}
We do not want to work out a convergence result corresponding to Theorem \ref{3948560384765083746534} for this coupled problem. We believe however that the analogous statement indeed holds true.
\end{remark}

\section*{Acknowledgements} 
This article was written during the Trimester program "Evolution of Interfaces" at the Hausdorff Research Institute for Mathematics (HIM), University of Bonn, in 2019. The support and the hospitality of HIM are gratefully acknowledged.
\bibliographystyle{plain}
\bibliography{bibo}

\begin{thebibliography}{10}

\bibitem{mulsekpaper12}
H.~Abels, M.~Rauchecker, and M.~Wilke.
\newblock Well-{P}osedness and qualitative behaviour of the {M}ullins-{S}ekerka
  problem with ninety-degree angle boundary contact, 2019.
\newblock \url{http://arxiv.org/abs/1902.03611}.

\bibitem{abelswilke}
H.~Abels and M.~Wilke.
\newblock Well-posedness and qualitative behaviour of solutions for a two-phase
  {N}avier-{S}tokes-{M}ullins-{S}ekerka system.
\newblock {\em Interfaces and Free Boundaries}, 15:39--75, 2013.

\bibitem{amann}
H.~Amann.
\newblock Nonhomogeneous linear and quasilinear elliptic and parabolic boundary
  value problems.
\newblock {\em Function Spaces, Differential Operators and Nonlinear Analysis},
  pages 9--126, 1993.

\bibitem{garcke232}
J.W. Barrett, H.~Garcke, and R.~N\"urnberg.
\newblock Parametric finite element approximations of curvature driven
  interface evolutions, 2019.
\newblock \url{http://arxiv.org/abs/1903.09462}.

\bibitem{bourgain1984}
J.~Bourgain.
\newblock Extension of a result of {B}enedek, {C}alderon and {P}anzone.
\newblock {\em Ark. Mat.}, 22(1-2):91--95, 12 1984.

\bibitem{engelnagel}
K.-J. Engel and R.~Nagel.
\newblock {\em {O}ne-parameter semigroups for linear evolution equations}.
\newblock Springer, New York, 2000.

\bibitem{mulsek2d}
H.~Garcke and M.~Rauchecker.
\newblock Stability analysis for stationary solutions of the
  {M}ullins-{S}ekerka flow with boundary contact, 2019.
\newblock \url{https://arxiv.org/abs/1907.00833}.

\bibitem{johnsensickel2008}
J.~Johnsen and W.~Sickel.
\newblock On the trace problem for {L}izorkin-{T}riebel spaces with mixed
  norms.
\newblock {\em Mathematische Nachrichten}, 281(5):669--696, 2008.

\bibitem{kaipdiss}
M.~Kaip.
\newblock {\em General parabolic mixed order systems in $L^p$ and
  applications}.
\newblock PhD thesis, Universit\"at Konstanz, 2012.

\bibitem{lunardioptimal}
A.~Lunardi.
\newblock {\em Analytic semigroups and {O}ptimal {R}egularity in {P}arabolic
  {P}roblems}.
\newblock Springer, 1995.

\bibitem{MR2911497}
M.~Meyries and M.~Veraar.
\newblock Sharp embedding results for spaces of smooth functions with power
  weights.
\newblock {\em Studia Math.}, 208(3):257--293, 2012.

\bibitem{pruessbuch}
J.~Pr\"uss and G.~Simonett.
\newblock {\em Moving interfaces and quasilinear parabolic evolution
  equations}.
\newblock Birkh\"auser Verlag, 2016.

\bibitem{navmulsekpap}
M.~Rauchecker and M.~Wilke.
\newblock Well-posedness and qualitative behaviour of a two-phase
  {N}avier-{S}tokes/{M}ullins-{S}ekerka system with ninety degree angle
  boundary contact.
\newblock 2019.

\bibitem{wilkehabil}
M.~Wilke.
\newblock Rayleigh-{T}aylor instability for the two-phase {N}avier-{S}tokes
  equations with surface tension in cylindrical domains, 2013.
\newblock Habilitationsschrift, Universit\"at Halle. Available online at
  \url{http://arxiv.org/abs/1703.05214}.

\end{thebibliography}
\end{document}